\documentclass[10pt]{article}

\usepackage[T1]{fontenc}
\usepackage[latin1]{inputenc}
\usepackage{amssymb}
\usepackage{amsfonts}
\usepackage{amsmath,mathtools}
\usepackage{graphicx}
\usepackage{color}
\usepackage{esint}
\usepackage{verbatim}
\newenvironment{proof}[1][Proof]{\noindent\textbf{#1.} }{\ \rule{0.5em}{0.5em}}
\newtheorem{theorem}{Theorem}[section]
\newtheorem{corollary}{Corollary}[section]
\newtheorem{lemma}{Lemma}[section]
\newtheorem{proposition}{Proposition}[section]

\newtheorem{definition}{Definition}[section]
\newtheorem{remark}{Remark}[section]

\begin{document}

\title{Reiterated periodic homogenization of integral functionals with convex and nonstandard growth integrands.}

\author{{\sc Joel Fotso Tachago}\\
	University of Bamenda \\ Faculty of Science, Higher Teachers Training College  Mathematics department,\\ P.O. Box 39, Bambili, Cameroon \\ fotsotachago@yahoo.fr,
	\\
	{\sc Hubert Nnang} \\ University of Yaounde I, \\
	\'{E}cole Normale Sup\'{e}rieure de Yaound\'{e}, P.O. Box 47 Yaound\'e, Cameroon, \\hnnang@uy1.uninet.cm. hnnang@yahoo.fr
	and\\
	{\sc Elvira Zappale} \\
	Dipartimento di Ingegneria Industriale 
	\\
	Universit\`a degli Studi di Salerno\\
	Via Giovanni Paolo II, 132\\
	84084 Fisciano (SA) (Italy),\\
	ezappale@unisa.it}
\date{February, 2020}
\maketitle

\begin{abstract}
Multiscale periodic homogenization is extended to an Orlicz-Sobolev setting. It
is shown by the reiteraded periodic two-scale convergence method that the
sequence of minimizers of a class of highly oscillatory minimizations
problems involving convex functionals, converges to the  minimizers of a
homogenized problem with a suitable convex function.
\end{abstract}
\bigskip\noindent{{\it Keywords: Convex function, reiterated two-scale convergence, relaxation, Orlicz Sobolev spaces.}
	\\
{\bf 35B27, 35B40, 35J25, 46J10, 49J45}} 

\section{Introduction\label{sec1}}

\bigskip The method of two-scale convergence introduced by Nguetseng \cite
{ngu0} and later developed by Allaire \cite{All1} have been widely adopted in homogenization of PDEs in classical Sobolev spaces
neglecting materials where microstructure cannot be conveniently captured by modeling exclusively by means of thoses spaces. Recently in \cite{fotso nnang 2012} some of the above methods were extended to Orlicz-Sobolev setting. On the other hand, an increasing
number of works in homogenization and dimension reduction (see \cite{Kalousek, kenne Nnang, Elvira2, Elvira3, LN1, LN2, LN3,  Nnang Orlicz 2014}, among the others)
are devoted to deal with this more general setting.   We also refer to \cite{ZP1, ZP2, ZP3} for two scale homogenization in variable exponent spaces, which also evidence Lavrentieff phenomena. 


In order to model multiscale phenomena, i.e., to provide homogenization results closer to reality, more than two-scales should be
considered. Indeed the aim of this work is to show that the two-scale
convergence method can be extended and generalized to tackle reiterated homogenization problems in the Orlicz-Sobolev setting.

In details, we intend to study the asymptotic behaviour as $\varepsilon \to 0^+$ of the sequence of
solutions of the problem 
\begin{equation}
\min \left\{ F_{\varepsilon }\left( v\right) :v\in W_{0}^{1}L^{B}\left(
\Omega \right) \right\}  \label{PR1}
\end{equation}%
where, for each $\varepsilon >0,$ the functional $F_{\varepsilon }$ is
defined on $W_{0}^{1}L^{B}\left( \Omega \right) $ by 
\begin{equation}\label{Fepsilon}
F_{\varepsilon }\left( v\right) =\int_{\Omega }f\left( \frac{x}{\varepsilon }
,\frac{x}{\varepsilon ^{2}},Dv\left( x\right) \right) dx, \; v\in
W_{0}^{1}L^{B}(\Omega),
\end{equation}
$\Omega $ being a bounded open set in $\mathbb{R}_x^N, n, N\in \mathbb N$, $D$ denoting the
gradient operator in $\Omega $ with respect to $x$ and the function
$f: \mathbb{R}
_{y}^{N}\times 
\mathbb{R}
_{z}^{N}\times 
\mathbb{R}
^{nN}\rightarrow \lbrack 0,+\infty )$ being an integrand, that satisfies the
following hypotheses:

$\left( H_{1}\right) $ for all $\lambda \in 
\mathbb{R}
^{N},$ $f\left( \cdot,z,\lambda \right) $ is measurable for all $z\in 
\mathbb{R}
^{N}$ and $f\left( y,\cdot ,\lambda \right) $ is continuous for almost all $y\in 
\mathbb{R}
^{N}$;

$\left( H_{2}\right) $ $f\left( y,z,\cdot\right) $ is strictly convex for a.e. $ y\in 
\mathbb R_y^N$ and all $z\in 
\mathbb R _z^{N}$;

$\left( H_{3}\right) $ for each $\left( k,k^{\prime }\right) \in 
\mathbb{Z}
^{2N}$ we have $f\left( y+k,z+k^{\prime },\lambda \right) =$ $f\left(
y,z,\lambda \right) $ for all $\left( z,\lambda \right) \in 
\mathbb{R}
_{z}^{N}\times 
\mathbb{R}
^{N}$ and a.e. $y\in 
\mathbb R
_{y}^{N}$;

$\left( H_{4}\right) $ there exist two constants $c_{1},c_{2}>0$ such that:
\begin{equation*}
c_{1}B\left( \left\vert \lambda \right\vert \right) \leq f\left( y,z,\lambda
\right) \leq c_{2}\left( 1+B\left( \left\vert \lambda \right\vert \right)
\right) 
\end{equation*}%
for all $\lambda \in \mathbb R
^{nN}$and for a.e. $y\in 
\mathbb R_y^N$ and all $z\in 
\mathbb R_z^N.$

We observe that problems of the type \eqref{PR1} have been studied by many
authors in many contexts (see, among the others, \cite{All1, All2, Amar, BabadjianBaia, BaiaFonseca, Barchiesi, CRZ, CDDA, CDG, FerreiraFonseca1, FerreiraFonseca2, Elvira 1, FTN2014, Neukamm, Visintin}.
But in all the above papers the
two-scale approach or other methods (see in particular unfolding) have been always considered in classical Sobolev setting. The novelty here is the multiscale approach beyond classical Sobolev
spaces. For the sake of exposition we consider the scales $\varepsilon$ and $\varepsilon^2$, but more general choices are possible, as in \cite{All2}.

In particular we introduce the following setting. 

Let $B$ an ${\rm N}-$function and $%
\widetilde{B}$ its conjugate both verifying the $\triangle_{2}$ condition, let $
\Omega $ be a bounded open set in $\mathbb R_{x}^{N},$ $Y=Z=\left(-\frac{1}{2},\frac{1}{2}\right) ^{N},  N\in \mathbb{N}$ and $\varepsilon$ any sequence of positive numbers converging to $0$. Assume that $\left( u_{\varepsilon }\right)_\varepsilon$ is bounded in $W^{1}L^{B}\left( \Omega \right).$ Then, there exist not relabelled subsequences  $\varepsilon, (u_\varepsilon)_\varepsilon, u_{0}\in W^{1}L^{B}\left( \Omega \right),$ 
\begin{equation*}
\left( u_{1},u_{2}\right) \in L^{1}\left( \Omega ;W_{\#}^{1}L^{B}\left(
Y\right) \right) \times L^{1}\left( \Omega ;L^1_{per}\left(
Y;W_{\#}^{1}L^{B}\left( Z\right) \right) \right)
\end{equation*}
such that: $u_{\varepsilon }\rightharpoonup u_{0}$ in $W^{1}L^{B}\left( \Omega
\right) $ weakly, and 
\begin{equation*}
\int_{\Omega }D_{x_{i}}u_{\varepsilon }\varphi \left( x,\frac{x}{\varepsilon 
},\frac{x}{\varepsilon ^{2}}\right) dx\rightarrow \iiint_{\Omega
\times Y\times Z}\left( D_{x_{i}}u_{0}+D_{y_{i}}u_{1}+D_{z_{i}}u_{2}\right)
\varphi \left( x,y,z\right) dxdydz
\end{equation*}%
$1\leq i\leq N$, and for all $%
\varphi \in L^{\widetilde{B}}\left( \Omega ;\mathcal{C}_{per}\left( Y\times
Z\right) \right),$ where $D_{x_i}, D_{y_i}$ and $D_{z_i}$ denote the distributional derivatives with respect to the variables $x_i.y_i,z_i$, also denoted by $\frac{\partial}{\partial_{x_i}}$, $\frac{\partial}{\partial_{y_i}}$ and $\frac{\partial}{\partial_{z_i}},$ respectively  (see Section \ref{notations} for detailed notations and Definition \ref{def3s} and Proposition \ref{mainprop3s} for rigorous results).

Next, we define, following the same type of notation adopted in \cite{fotso nnang 2012}, the space
\begin{equation}
\label{F01}
\mathbb{F}_{0}^{1}L^{B}=W_{0}^{1}L^{B}(\Omega) \times L_{D_{y}}^{B}\left( \Omega ;W_{\#}^{1}L^{B}(Y) \right) \times L_{D_{z}}^{B}\left( \Omega ;L_{per}^1\left(
Y;W_{\#}^{1}L^{B}(Z) \right) \right),
\end{equation} where%
\begin{align}
L_{D_{y}}^{B}\left( \Omega ;W_{\#}^{1}L^{B}(Y) \right) =\left\{ u\in L^{1}\left( \Omega
;W_{\#}^{1}L^{B}(Y) \right) :D_{y}u\in L^{B}_{per}\left( \Omega \times Y\right)
^{N}\right\},\nonumber \\
L_{D_{z}}^{B}\left( \Omega ;L_{per}^1\left( Y;W_{\#}^{1}L^{B}(Z) \right) \right) = \label{LF10}\\ 
\left\{ u\in L^{1}\left( \Omega ;L_{per}^1\left( Y;W_{\#}^{1}L^{B}(
Z) \right) \right) :D_{z}u\in L_{per}^{B}\left( \Omega \times Y\times
Z\right) ^{N}\right\}.\nonumber 
\end{align}
 Observe that $D_x, D_y$ and $D_z$ denote the vector of distributional derivatives with respect to $x\equiv(x_1,\dots, x_N)$, $y\equiv(y_!,\dots,y_N)$ and $z\equiv(z_1,\dots,z_N)$ respectively. 

We equip $\mathbb{F}_{0}^{1}L^{B}$ with the norm $\left\Vert u\right\Vert _{%
	\mathbb{F}_{0}^{1}L^{B}}=\left\Vert Du_{0}\right\Vert _{B,\Omega
}+\left\Vert D_{y}u_{1}\right\Vert _{B,\Omega \times Y}+\left\Vert
D_{z}u_{2}\right\Vert _{B,\Omega \times Y\times Z}$, $u=\left(
u_{0},u_{1},u_{2}\right) \in \mathbb{F}_{0}^{1}L^{B}$ which makes it a
Banach space. 

 Finally for $v=\left(
v_{0},v_{1},v_{2}\right) \in \mathbb{F}_{0}^{1}L^{B}$, denote by $\mathbb{D}%
v$, the sum $Dv_{0}+D_{y}v_{1}+D_{z}v_{2}$ and define the functional $F:\mathbb{F}%
_{0}^{1}L^{B} \to \mathbb R^+$ by 
\begin{equation}
\label{F}
F\left( v\right) =\iiint_{\Omega \times Y\times
	Z}f\left( \cdot,\mathbb{D}v\right) dxdydz.
\end{equation}
With the tool of multiscale convergence at hand in the Orlicz-Sobolev setting, we prove 

\begin{theorem}\label{main} Let $\Omega $ be a bounded open set in $\mathbb R^N_x$ and let $f: \mathbb{R}
_{y}^{N}\times 
\mathbb{R}
_{z}^{N}\times 
\mathbb{R}
^{N}\rightarrow \lbrack 0,+\infty )$ be an integrand satisfying $(H_1)-(H_4)$. 
	For each $\varepsilon >0,$ let $u_{\varepsilon }$ be the unique solution of
	\eqref{PR1}, then as $\varepsilon \rightarrow 0,$
	\begin{itemize}
		\item [(a)]$u_{\varepsilon }\rightharpoonup u_{0}$ weakly in $W_{0}^{1}L^{B}(\Omega)$;
		\item[(b)]
 $Du_{\varepsilon }\rightharpoonup \mathbb{D
}u=Du_{0}+D_{y}u_{1}+D_{z}u_{2}$  weakly reiteratively two-scale in $L^{B}\left( \Omega \right)^{N}-$, where $u=\left( u_{0},u_{1},u_{2}\right) \in \mathbb{F}%
_{0}^{1}L^{B}$ is the unique solution of the minimization problem \begin{equation}
F\left( u\right) =\underset{v\in \mathbb{F}_{0}^{1}L^{B}}{\min }F\left(
v\right),  \label{MP1}
\end{equation}%
	\end{itemize}
	where $\mathbb F_0^1L^B$ and $F$ are as in \eqref{F01} and \eqref{F}, respectively.
\end{theorem}

The paper is organized as follows, Section 2 deals with notations, preliminary results on Orlicz-Sobolev spaces, introduction of suitable function spaces to deal with multiple scales homogenization, and compactness result for reiterated two-scale convergence, while Section 3 contains the main results devoted to the proof of Theorem \ref{main}, together with Corollary \ref{maincor2} which allows to recast the main result in the framework of {$\Gamma$}\color{black} convergence (see also \cite{FTNZIMSE} for the single scale case).

\section{Notation and Preliminaries}\label{notations}

In what follows $X$ and
$V$ denote a locally compact space and a Banach space, respectively, and
$C(X; V)$ stands for the space of continuous functions from $X$ into $V$ , and
$C_b(X; V)$ stands for those functions in $C(X; F)$ that are bounded. The space $C_b(X; V)$ is enodowed with the supremum norm $\|u\|_{\infty} = \sup_{x\in X}
\|u(x)\|$ , where
$\|·\|$ denotes the norm in $V$, (in particular, given an open set $A\subset \mathbb R^N$ by $\mathcal C_b(A)$ we denote the space of real valued continuous and bounded functions defined in $A$). Likewise the spaces $L^p(X; V)$ and $L^p_{\rm loc}(X; V)$
($X$ provided with a positive Radon measure) are denoted by $L^p(X)$ and
$L^p_{\rm loc}(X)$, respectively, when $V = \mathbb R$ (we refer to \cite{bour, bour2, edw} for integration theory).

In the sequel we denote by $Y$ and $Z$ two identical copies of the cube $]-1/2,1/2[^N$.  

In order to enlighten the space variable under consideration we will adopt the notation $\mathbb R^N_x, \mathbb R^N_y$, or $\mathbb R^N_z$ to indicate where $x,y $ or $z$ belong to.
 
The family of open subsets in $\mathbb R^N_x$ will be denoted by $\mathcal A(\mathbb R^N_x)$.

For any subset $E$ of $\mathbb R^m$, $m \in \mathbb N$, by $\overline E$, we denote its closure in the relative topology.

For every $x \in \mathbb R^N$ we denote by $[x]$ its integer part, namely the vector in $\mathbb Z^N$, which has as component the integer parts of the components of $x$.

  By $\mathcal L^N$ we denote the Lebesgue measure in $\mathbb R^N$.
\subsection{Orlicz-Sobolev spaces}

\bigskip Let $B:\left[ 0,+\infty \right[ \rightarrow \left[ 0,+\infty \right[
$ be an ${\rm N}-$function \cite{ada}, i.e., $B$ is continuous, convex, with $%
B\left( t\right) >0$ for $t>0,\frac{B\left( t\right) }{t}\rightarrow 0$ as $t\rightarrow 0,$ and $\frac{B\left( t\right) }{t}\rightarrow \infty $ as $%
t\rightarrow \infty .$
Equivalently, $B$ is of the form $B\left( t\right)
=\int_{0}^{t}b\left( \tau \right) d\tau ,$ where $b:\left[ 0,+\infty \right[
\rightarrow \left[ 0,+\infty \right[ $ is non decreasing, right continuous,
with $b\left( 0\right) =0,b\left( t\right) >0$ if $t>0$ and $b\left(
t\right) \rightarrow +\infty $ if $t\rightarrow +\infty .$ 

We denote by $\widetilde{B},$ the complementary ${\rm N}-$function of $B$ defined by $\widetilde{B}(t)=\sup_{s\geq 0}\left\{ st-B\left( s\right) ,t\geq 0\right\} $ . It follows
that 
\begin{equation}\nonumber 
\frac{tb(t)}{B(t)} \geq 1
\;\;(\hbox{or }> \hbox{if }b\hbox{ is strictly increasing}),
\end{equation}
\begin{equation}
\nonumber \widetilde{B}( b(t) )\leq
tb( t) \leq B( 2t) \hbox{ for all }t>0.
\end{equation}
An ${\rm N}-$function $B$ is of class $\triangle _{2}$ (denoted $B\in \triangle
_{2}$) if there are $\alpha >0$ and $t_{0}\geq 0$ such that $B\left(
2t\right) \leq \alpha B\left( t\right) $ for all $t\geq t_{0}$. 

\noindent In all what
follows $B$ and $\widetilde{B}$ are conjugates ${\rm N}-$function$s$
satisfying the $\triangle_2$ (delta-2) condition and $c$ refer to a constant. Let $\Omega $ be a
bounded open set in $\mathbb R^N, (N \in \mathbb N).$ The Orlicz-space
 $$L^{B}\left(
\Omega \right) =\left\{ u:\Omega \rightarrow 
\mathbb R \hbox{ measurable},\lim_{\delta \to 0^+} \int_{\Omega
}B\left( \delta \left\vert u\left( x\right) \right\vert \right) dx=0\right\} 
$$ 
is a Banach space for the Luxemburg norm: $$\left\Vert u\right\Vert
_{B,\Omega }=\inf \left\{ k>0:\int_{\Omega }B\left( \frac{\left\vert u\left(
x\right) \right\vert }{k}\right) dx\leq 1\right\} <+\infty .$$It follows
that: $\mathcal{D}\left( \Omega \right) $ is dense in $L^{B}\left( \Omega
\right), L^{B}\left( \Omega \right) $ is separable and reflexive, the dual
of $L^{B}\left( \Omega \right) $ is identified with $L^{\widetilde{B}}\left(
\Omega \right),$ and the norm on $L^{\widetilde{B}}\left( \Omega \right) $
is equivalent to $\left\Vert \cdot\right\Vert _{\widetilde{B},\Omega }.$
We will denote the norm of elements in $L^{B}\left( \Omega \right)$, both by $\|\cdot\|_{L^{B}\left( \Omega \right)}$ and with $\|\cdot\|_{B, \Omega}$, the latter symbol being useful when we want emphasize the domain $\Omega$.

Futhermore, it is also convenient to recall that:
\begin{itemize} 
	\item[(i)] $
\left\vert \int_{\Omega }u\left( x\right) v\left( x\right) dx\right\vert
\leq 2\left\Vert u\right\Vert _{B,\Omega }\left\Vert v\right\Vert _{%
\widetilde{B},\Omega }$ for $u\in L^{B}\left( \Omega \right) $ and $v\in L^{%
\widetilde{B}}\left( \Omega \right) $, 
\item[(ii)] given $v\in L^{%
\widetilde{B}}\left( \Omega \right) $ the linear functional $L_{v}$ on $%
L^{B}\left( \Omega \right) $ defined by $L_{v}\left( u\right) =\int_{\Omega
}u\left( x\right) v\left( x\right) dx,\left( u\in L^{B}\left( \Omega \right)
\right) $ belongs to the dual $\left[ L^{B}\left( \Omega \right) \right]
^{\prime }=L^{\widetilde{B}}\left( \Omega \right) $ with $\left\Vert
v\right\Vert _{\widetilde{B},\Omega }\leq \left\Vert L_{v}\right\Vert _{\left[ L^{B}\left( \Omega \right) \right] ^{\prime }}\leq 2\left\Vert
v\right\Vert _{\widetilde{B},\Omega }$, 
\item[(iii)]  the property $\lim_{t \to +\infty} \frac{B\left( t\right) }{t}=+\infty $
implies $L^{B}\left( \Omega \right) \subset L^{1}\left( \Omega \right)
\subset L_{loc}^{1}\left( \Omega \right) \subset \mathcal{D}^{\prime }\left(
\Omega \right),$ each embedding being continuous.
\end{itemize}

For the sake of notations, given any $d\in \mathbb N$, when $u:\Omega \to \mathbb R^d$, such that each component $(u^i)$, of $u$, lies in $L^B(\Omega)$  
we will denote the norm of $u$ with the symbol $\|u\|_{L^B(\Omega)^{d}}:=\sum_{i=1}^d \|u^i\|_{B,\Omega}$.

Analogously one can define the Orlicz-Sobolev functional space as follows: 

\noindent$%
W^{1}L^{B}\left( \Omega \right) =\left\{ u\in L^{B}\left( \Omega \right) :%
\frac{\partial u}{\partial x_{i}}\in L^{B}\left( \Omega \right) ,1\leq i\leq
d\right\} ,$ where derivatives are taken in the distributional sense on $%
\Omega .$ Endowed with the norm $\left\Vert u\right\Vert _{W^{1}L^{B}\left(
\Omega \right) }=\left\Vert u\right\Vert _{B,\Omega }+\sum_{i=1}^{d}$ $%
\left\Vert \frac{\partial u}{\partial x_{i}}\right\Vert _{B,\Omega },u\in
W^{1}L^{B}\left( \Omega \right) ,$\ \ $W^{1}L^{B}\left( \Omega \right) $ is
a reflexive Banach space. We denote by $W_{0}^{1}L^{B}\left( \Omega \right)
, $ the closure of $\ \mathcal{D}\left( \Omega \right) $\ in $%
W^{1}L^{B}\left( \Omega \right) $ and the semi-norm $u\rightarrow \left\Vert
u\right\Vert _{W_{0}^{1}L^{B}\left( \Omega \right) }=\left\Vert
Du\right\Vert _{B,\Omega }=\sum_{i=1}^{d}$ $\left\Vert \frac{\partial u}{%
\partial x_{i}}\right\Vert _{B,\Omega }$ is a norm on $W_{0}^{1}L^{B}\left(
\Omega \right) $ equivalent to $\left\Vert \cdot \right\Vert _{W^{1}L^{B}\left(
\Omega \right) }.$

By $W_{\#}^{1}L^{B}\left( Y\right)$, we denote the space of functions $u \in W^1L^B(Y)$ such that $\int_Y u(y)d y=0$.  It is endowed with the gradient norm.
Given a function space $S$ defined in $Y$, $Z$ or $Y\times Z$, the subscript $S_{per}$ means that the functions are periodic in $Y$, $Z$ or $Y\times Z$, as it will be clear from the context. In particular $C_{per}(Y\times Z)$ denote the space of periodic functions in $C(\mathbb R^N_y\times \mathbb R^N_z)$, i.e. that verify $w(y + k, z + h) = w(y, z)$ for $(y, z) \in \mathbb R^N \times \mathbb R^N$
and $(k, h) \in \mathbb Z^N × \mathbb Z^N$. $C^\infty_{per}(Y\times Z)=C_{per}(Y\times Z)\cap C^\infty(\mathbb R^N_y\times \mathbb R^N)$, and $L^p
_{per}(Y \times Z)$ is the space of
$Y \times Z$ -periodic functions in $L^p_{loc}(\mathbb R^N_y
\times \mathbb R^N_z)$.

\subsection{Fundamentals of reiterated homogenization in Orlicz spaces}


This subection is devoted to show some results which are useful for an explicit construction of reiterated multiscale convergence in the Orlicz setting. Indeed all the definitions are given starting from spaces of regular functions, then several norms are introduced together with proofs of functions spaces' properties. 
	On the other hand we will not present neither arguments which are very similar to the ones used to deal with standard two scale convergence in the Orlichz setting, nor those related to reiterated two-scale convergence in the standard Sobolev setting (for the latter we refer to \cite[Sections 2 and 4]{nnang reit}). 

We start by defining rigorously the traces of the form $u\left( x,\frac{x}{\varepsilon 
},\frac{x}{\varepsilon ^{2}}\right) ,x\in \Omega,$ $\varepsilon >0$. 
We will consider several cases, according to the regularity of $u$.

\emph{Case 1: }$u\in \mathcal{C}\left( \Omega \times 
\mathbb R_y^N\times
\mathbb R_z^N\right) $

We define 
 \begin{equation*}
 u^{\varepsilon }\left(
x\right) :=u\left( x,\frac{x}{\varepsilon },\frac{x}{\varepsilon ^{2}}\right)
\end{equation*}
Obviously $u^{\varepsilon }\in \mathcal{C}\left( \Omega
\right) .$ We define the trace operator of order $\varepsilon
>0,(t_{\varepsilon })$ by 
\begin{equation}
\label{traceoperator}t_{\varepsilon }:u \in \mathcal{C}\left( \Omega
\times 
\mathbb R
_{y}^{N}\times 
\mathbb R_z^N\right) \longrightarrow u^\varepsilon \in \mathcal{C}\left( \Omega \right).
\end{equation}

It results that the operator $t^{\varepsilon }$ in \eqref{traceoperator} is linear and continuous.
\color{black}

\medskip
\emph{Case 2: }$u\in \mathcal{C}\left( \overline{\Omega };\mathcal C_b\left( 
\mathbb R_{y}^{N}\times 
\mathbb R_z^N\right) \right) $.

$\mathcal{C}\left( \overline{\Omega };\mathcal C_b\left( 
\mathbb R_y^N\times \mathbb R
_z^N\right) \right) \subset \mathcal{C}\left( \overline{\Omega };%
\mathcal{C}\left(
\mathbb R_y^N\times 
\mathbb R_z^N\right) \right) \widetilde{=}\mathcal{C}\left( \overline{\Omega }%
\times 
\mathbb R_y^N\times
\mathbb R_z^N\right) .$ We can then consider $\mathcal{C}\left( \overline{\Omega }%
;\mathcal C_b\left( 
\mathbb R_y^N\times 
\mathbb R_z^N\right) \right) $ as a subspace of $\mathcal{C}\left( \overline{%
\Omega }\times 
\mathbb R
_y^N\times 
\mathbb R_z^N\right) $. Since $\overline{\Omega }$\ \ is compact in $\mathbb R_{x}^{N},$ then $u^{\varepsilon }\in \mathcal C_b\left( \Omega \right) $ and
the above operator can be considered from $\mathcal{C}\left( \overline{\Omega };%
\mathcal C_b\left(
\mathbb{R}
_{y}^{N}\times 
\mathbb{R}
_{z}^{N}\right) \right) $ to $\mathcal C_b\left( \Omega \right),$
as linear and continuous.
\medskip

\emph{Case 3: }$u\in L^{B}( \Omega ;V)$ where $V$ is a closed vector subspace of $\mathcal C_b\left( \mathbb R _y^N\times 
\mathbb R_z^N\right).$

Recall that $u\in L^{B}(\Omega ;V) $ means the function $x\rightarrow \left\Vert
u\left( x\right) \right\Vert _{\infty }\ \ \ $from $\Omega $ into $
\mathbb R 
$ belongs to $L^{B}\left( \Omega \right) $ and
\begin{equation*}
\left\Vert u\right\Vert _{L^{B}\left( \Omega ;\mathcal C_b\left(\mathbb R
_{y}^{N}\times 
\mathbb R_{z}^{N}\right) \right) }=\inf \left\{ k>0:\int_{\Omega }B\left( \frac{
\left\Vert u\left( x\right) \right\Vert _{\infty }}{k}\right) dx\leq
1\right\} <+\infty .
\end{equation*}

Let $u\in \mathcal{C}\left( \overline{\Omega };C_b\left( 
\mathbb{R}
_{y}^{N}\times 
\mathbb{R}
_{z}^{N}\right) \right) ,$ then $\left\vert u^{\varepsilon }\left( x\right)
\right\vert =\left\vert u\left( x,\frac{x}{\varepsilon },\frac{x}{%
\varepsilon ^{2}}\right) \right\vert \leq \left\Vert u\left( x\right)
\right\Vert _{\infty }.$ As ${\rm N}-$functions are non decreasing we deduce that: 
\begin{equation*}
B\left( \frac{\left\vert u^{\varepsilon }\left( x\right) \right\vert }{k}%
\right) \leq B\left( \frac{\left\Vert u\left( x\right) \right\Vert _{\infty }%
}{k}\right) ,\hbox{ for all } k>0,\hbox{ for all } x\in \overline{\Omega }.
\end{equation*}
Hence we get $\int_{\Omega }B\left( \frac{\left\vert u^{\varepsilon }\left(
x\right) \right\vert }{k}\right) dx\leq \int_{\Omega }B\left( \frac{%
\left\Vert u\left( x\right) \right\Vert _{\infty }}{k}\right) dx,$ thus $%
\int_{\Omega }B\left( \frac{\left\Vert u\left( x\right) \right\Vert _{\infty
}}{k}\right) dx\leq 1\Longrightarrow \int_{\Omega }B\left( \frac{\left\vert
u^{\varepsilon }\left( x\right) \right\vert }{k}\right) dx\leq 1,$ that is, 
$$\left\Vert u^{\varepsilon }\right\Vert _{L^{B}\left( \Omega \right) }\leq
\left\Vert u\right\Vert _{L^{B}\left( \Omega ;\mathcal C_b\left( 
\mathbb{R}
_{y}^{N}\times 
\mathbb{R}
_{z}^{N}\right) \right) }.$$ Therefore the trace operator $u\rightarrow
u^{\varepsilon }$ from $\mathcal{C}\left( \overline{\Omega };V \right) $ into $L^{B}\left( \Omega \right) ,$ extends by
density and continuity to a unique operator from $L^{B}( \Omega ;\mathcal C_b(V)) $.

It will be still denoted  by $$ t^\varepsilon :
u\rightarrow u^{\varepsilon }$$ and it verifies: \begin{equation}
\label{tracebounds}\left\Vert u^{\varepsilon
}\right\Vert _{L^{B}\left( \Omega \right) }\leq \left\Vert u\right\Vert
_{L^{B}\left( \Omega ;\mathcal C_b\left( 
\mathbb{R}
_{y}^{N}\times 
\mathbb{R}
_{z}^{N}\right) \right) }, \hbox{ for all }u\in L^{B}\left( \Omega ;(V) \right) .
\end{equation}
\noindent In order to deal with reiterated multiscale convergence we need to have good definition for the measurability of test functions, so we should ensure measurability for the trace of elements $u\in L^{\infty }\left( 
\mathbb{R}
_{y}^{N};\mathcal C_b\left( 
\mathbb{R}
_{z}^{N}\right) \right) $ and $u\in \mathcal{C}\left( \overline{\Omega }%
;L^{\infty }\left( 
\mathbb{R}
_{y}^{N};\mathcal C_b\left( 
\mathbb{R}
_{z}^{N}\right) \right) \right) $, but we omit these proofs, referring to \cite[Section 2]{nnang reit}.

Let $M:\mathcal C_{per}\left( Y\times Z\right) \rightarrow 
\mathbb R$ be the mean value functional (or equivalently 'averaging operator') defined as 
\begin{equation}
\label{M}
u\rightarrow M(u):=\iint_{Y\times Z}u\left( x,y\right) dxdy.
\end{equation}
It results that 
\begin{itemize}
\item[(i)] $M$ is nonnegative, i.e. $M\left( u\right) \geq 0\
\hbox{ for all } u\in \mathcal{C}_{per}(Y\times Z) ,u\geq 0;$
\item[(ii)] $M$ is continuous on $\mathcal{C%
}_{per}\left( Y\times Z\right) $ (for the sup norm); 
\item[(iii)]
$M\left( 1\right) =1$;
\item[(iv)]  $M$ is translation invariant.
\end{itemize}


\medskip
In the same spirit of \cite{nnang reit}, for the given ${\rm N}-$function $B$, we define $\Xi ^{B}\left( \mathbb{R}
_{y}^{N};\mathcal C_b\left( 
\mathbb{R}
_{z}^{N}\right) \right) $ or simply $\Xi ^{B}\left( 
\mathbb{R}
_{y}^{N};\mathcal C_b\right) $ the following space 
\begin{align}\label{Xi}
\Xi ^{B}\left( 
\mathbb{R}_{y}^{N};\mathcal C_b\right):=
\Big\{ u \in L_{loc}^{B}\left( 
\mathbb{R}
_{x}^{N};C_b\left( 
\mathbb{R}
_{z}^{N}\right) \right): \hbox{for every } U \in {\mathcal A}(\mathbb R^N_x): \nonumber\\
\left. \underset{0<\varepsilon \leq 1}{\sup }\inf \left\{
k>0,\int_{U}B\left( \frac{\left\Vert u\left( \frac{x}{\varepsilon },\cdot\right)
	\right\Vert _{L^{\infty }}}{k}\right) dx\leq 1\right\} <\infty \right\}.
\end{align}

 Hence putting 
\begin{equation}\label{originalnormLBPer}
\left\Vert u\right\Vert _{\Xi ^{B}\left( 
\mathbb{R}
_{y}^{N};\mathcal C_b\left( 
\mathbb{R}
_{z}^{N}\right) \right) }=\underset{0<\varepsilon \leq 1}{\sup }\inf \left\{
k>0,\int_{B_{N}(0,1)}B\left( \frac{\left\Vert u\left( \frac{x}{\varepsilon }%
,\cdot\right) \right\Vert _{L^{\infty }}}{k}\right) dx\leq 1\right\} ,
\end{equation}%
 with $B_N(0,1)$ being the unit ball of $
\mathbb{R}_x^N$ centered at the origin, we have a norm on \ $\Xi ^{B}\left( 
\mathbb{R}
_y^N;\mathcal C_b\left(
\mathbb{R}_z^N\right) \right) $ which makes it a Banach space. 

We also denote by $\mathfrak{X}_{per}
^{B}\left( 
\mathbb R _{y}^{N};\mathcal C_b\right) $ the closure of $\mathcal{C}_{per}\left(
Y\times Z\right) $\ in \ $\Xi ^{B}\left(
\mathbb R
_{y}^{N};\mathcal C_b\right) .$

Recall that $L_{per}^{B}\left( Y\times Z\right) $
denotes the space of functions in $L^B_{\rm loc}(\mathbb R_y^N\times \mathbb R^N_z)$ which are $Y\times Z$-periodic.

\noindent Clearly $\left\Vert \cdot\right\Vert _{B,Y\times Z}$ is a norm on $L_{per}^{B}\left(
Y\times Z\right) $, namely it suffices to consider the $L^B$ norm just on  the unit period.



Let $u\in \mathcal{C}_{per}\left( Y\times Z\right) $ , we have 
$$
\left\vert \int_{B_N(0,1)}u\left( \frac{x}{\varepsilon },\frac{x}{\varepsilon
	^{2}}\right) dx\right\vert \leq \int_{B_N(0,1)}\left\Vert u\left( \frac{x}{%
	\varepsilon },\cdot\right) \right\Vert _{\infty }dx\leq 2\left\Vert 1\right\Vert
_{\widetilde{B},B_N(0,1)}\left\Vert u\right\Vert _{\Xi^{B}\left(
	\mathbb R_y^N;\mathcal C_b\left( 
	\mathbb R_z^N\right) \right) }.$$

The following result, useful to prove estimates which involve test functions on oscillating arguments (see for instance Proposition \ref{propcomp}),  is a preliminary instrument which aims at comparing the $L^B$ norm in $Y\times Z$ with the one in \eqref{originalnormLBPer}.

\begin{lemma}\label{lemma2.2}
 There exists $C\in \mathbb R^+$ such that $\left\Vert u^{\varepsilon }\right\Vert
_{B,B_N(0,1)}\leq C \left\Vert u\right\Vert
_{B,Y\times Z },$ for every $0<\varepsilon \leq 1$, and $u\in \mathfrak{X}_{per}^{B}\left( 
\mathbb R
_{y}^{N};\mathcal C_b\right),$ 
\end{lemma}
\color{black}
\begin{proof}
Let $\varepsilon >0$.
We start observing that  we can always find a compact set $H \subset \mathbb R^N$ (independent on $\varepsilon$) such that 
$$B_N(0,1)\subseteq \cup_{k \in Z_{\varepsilon^2}} \varepsilon^2 (k+ Z) \subseteq H$$
where 
$Z_{\varepsilon^2}=\left\{k \in \mathbb Z^N: \varepsilon^2 (k+ Z)\cap \overline{B_N(0,1)}\not= \emptyset \right\}$. 

Define also $B_{N,\varepsilon^2}:={\rm int}\left(\bigcup_{k \in Z_{\varepsilon^2}} \varepsilon^2 (k+ \overline Z)\right)$. $B_N(0,1) \subset B_{N,\varepsilon^2}$.

Thus
\begin{align*}
\int_{B_N(0,1)} B\left(\left|u\left(\frac{x}{\varepsilon}, \frac{x}{\varepsilon^2} \right)\right|\right)dx \leq
\int_{\bigcup_{k \in Z_{\varepsilon^2}} \varepsilon^2 (k+ \overline Z)}B\left(\left|u\left(\frac{x}{\varepsilon}, \frac{x}{\varepsilon^2} \right)\right|\right)dx=\\
\displaystyle{\sum_{i=1}^{n(\varepsilon^2)} \varepsilon^{2N}\int_Z B\left(\left|u\left(\frac{\varepsilon^2 k_i + \varepsilon^2 z}{\varepsilon}, \frac{\varepsilon^2 k_i + \varepsilon^2 z}{\varepsilon^2} \right)\right|\right)dz=}\\
 \displaystyle{\sum_{i=1}^{n(\varepsilon^2)} \varepsilon^{2N}\int_Z B\left(\left|u\left(\varepsilon k_i + \varepsilon z, z\right)\right|\right)dz},
\end{align*}
where we have used the change of variables $x= \varepsilon^2 (k_i +z)$, in each cube $\varepsilon^2 (k_i+ Z)$, the periodicity of $u$ in the second variable, the fact that we can cover $B_N(0,1)$ with a finite number of cubes $\varepsilon^2(k_i+Z)$, depending on $\varepsilon^2$ and denoted by $n(\varepsilon^2)$.

Since $\left[\frac{x}{\varepsilon^2}\right]= k_i$ and $[z]=0$ for every $x \in \varepsilon^2( k_i+ Z)$ and $z\in Z$ and $\mathcal L^N(\varepsilon^2 (k_i+Z))=\varepsilon^{2N}$, we can write
\begin{align*}
\int_{B_N(0,1)} B\left(\left|u\left(\frac{x}{\varepsilon}, \frac{x}{\varepsilon^2} \right)\right|\right)dx \leq
\displaystyle{\sum_{i=1}^{n(\varepsilon^2)} \varepsilon^{2N}\int_Z B\left(\left|u\left(\varepsilon \left[\frac{x}{\varepsilon^2}\right] + \varepsilon z,z \right)\right|\right)dz\leq}\\
\displaystyle{\sum_{i=1}^{n(\varepsilon^2)} \int_{\varepsilon^2(k_i+Z)}\int_Z B\left(\left|u\left(\varepsilon \left[\frac{x}{\varepsilon^2}\right] + \varepsilon z,z \right)\right|\right)dzdx \leq}\\
\iint_{B_{N,\varepsilon^2}\times Z} B\left(\left|u\left(\varepsilon \left[\frac{x}{\varepsilon^2}\right] + \varepsilon z,z \right)\right|\right)dzdx=\\
\iint_{B_{N,\varepsilon^2}\times Z}B\left(\left|u\left( \frac{x}{\varepsilon}, z\right)\right|\right)dxdz,
\end{align*}
where in the third line above we have used the fact that $\frac{x}{\varepsilon}= \varepsilon \left[\frac{x}{\varepsilon^2}\right] + \varepsilon z$.

Now, making again another change of variable of the same type, i.e. $y+ h_i= x/\varepsilon$, after a covering of $B_{N,\varepsilon^2}$ made by $\bigcup_{h_i \in Z_{\varepsilon}}\varepsilon (h_i + Y)$, where $Z_{\varepsilon}=\left\{h \in \mathbb Z^N: \varepsilon (h+ Y)\cap \overline{B_{N,\varepsilon^2}}\not= \emptyset \right\}$ we have

\begin{align*}
\iint_{B_{N,\varepsilon^2}\times Z}B\left(\left|u\left( \frac{x}{\varepsilon}, z\right)\right|\right)dxdz\leq
\\
\displaystyle{\sum_{i=1}^{n(\varepsilon)}\varepsilon^N
	\iint_{{h_i +Y}\times Z}B\left(\left|u\left(\frac{\varepsilon h_i +\varepsilon y}{\varepsilon},z\right)\right|\right)dydz}\leq\\
\sum_{i=1}^{n(\varepsilon)}\varepsilon^N\iint_{Y\times Z} B\left(\left|u(y,z)\right|\right)dydz,
\end{align*}




Up to another choice of $0<\varepsilon_0\leq 1$, we can observe that, given $\varepsilon < \varepsilon_0$, $B_N(0,1)\subset B_{N,\varepsilon^2}$ and also $B_N(0,1)\subset \cup_{i=1}^{n(\varepsilon)} \varepsilon(h_i + Y)$. 
On the other hand there is a compact $H$, which contains $\cup_{i=1}^{n(\varepsilon)} \varepsilon(h_i + Y)$ and whose measure satisfies the following inequality ${\mathcal L}^N(H)\geq \sum_{i=1}^{n(\varepsilon)}\varepsilon^N$. 



Essentially repeating the same above computations, we have for every  $k \in \mathbb R^+$, and $0<\varepsilon \leq \varepsilon _{0}$ and $u\in L^B_{\rm per}(Y\times Z)$
:



\begin{equation*}
\int_{B_N(0,1)}B\left(\left|\frac{u\left( \frac{x}{\varepsilon },\frac{x}{\varepsilon ^{2}}\right) }{k}\right|\right) dx\leq \varepsilon ^{N}\sum_{i=1}^{n\left( \varepsilon \right) }\iint_{Y\times Z}B\left(\left| \frac{u\left( y,z\right) }{k}\right|\right) dydz.
\end{equation*}
For $k=\left\Vert u\right\Vert _{B,Y \times Z }$
using the convexity of $B$, and the fact that $B(0)=0$, we get: 
\begin{equation*}
\begin{array}{ll}
\int_{B_N(0,1)}B\left( \left|\frac{u\left( \frac{x}{\varepsilon },\frac{x}{%
\varepsilon ^{2}}\right) }{\left( 1+\mathcal L^N\left( H\right) \right) \left\Vert
u\right\Vert _{B,Y \times Z}}\right|\right) dx\leq \frac{1}{\left( 1+\mathcal L^N\left(
H\right) \right) }\int_{B_N(0,1)}B\left(\left| \frac{u\left( \frac{x}{\varepsilon },%
\frac{x}{\varepsilon ^{2}}\right) }{\left\Vert u\right\Vert _{B,Y \times Z }}%
\right|\right) dx\\
\displaystyle{\leq \varepsilon ^{N}\sum_{i=1}^{n\left( \varepsilon
\right) }\iint_{Y\times Z}B\left(\left| \frac{u\left( y,z\right) }{%
\left\Vert u\right\Vert _{B,Y \times Z }}\right|\right) dydz\times \frac{1}{\left(
1+\mathcal L^N \left( H\right) \right) }}\\
\displaystyle{\leq \frac{n\left( \varepsilon \right)
\varepsilon ^{N}}{\left( 1+\mathcal L^N\left( H\right) \right) } \iint_{Y\times Z}B\left(\left| \frac{u\left( y,z\right) }{%
\left\Vert u\right\Vert _{B,Y \times Z }}\right|\right) dydz}\\
\displaystyle{\leq \frac{\mathcal L^N\left( H\right) }{\left( 1+\mathcal L^N( H) \right) }}
\displaystyle{\iint_{Y\times Z}B\left(\left| \frac{u( y,z) }{\left\Vert u\right\Vert _{B,Y \times Z }}\right|\right) dydz< 1,}
\end{array}
\end{equation*}
where the non decreasing behavour of $B$ has been exploited. Therefore, by the definition of norm in $B_N(0,1)$, $\left\Vert u^{\varepsilon }\right\Vert _{B,B_N(0,1)}\leq \left(
1+\mathcal L^N\left( H\right) \right) \left\Vert u\right\Vert _{B,Y \times Z }$.
\end{proof}

\begin{lemma}
The mean value operator $M$ defined on $\mathcal{C}_{per}\left( Y\times Z\right) $ by \eqref{M}
can be
extended by continuity to a unique linear and continuous functional denoted in
the same way from $\mathfrak{X}_{per}^{B}\left( 
\mathbb R_y^N;\mathcal C_b\right) $\ to $
\mathbb R$ such that 
\begin{itemize}
	\item 	$M$ is non negative, i.e. for all $u\in \mathfrak{X}_{per}^{B}\left(\mathbb R_y^N;\mathcal C_b\right) ,u\geq 0\Longrightarrow M( u) \geq0,$
		\item $M$ is translation invariant.
\end{itemize} 
\end{lemma}

\begin{proof}
It is a consequence of the very defintions \eqref{Xi} and of $\mathfrak{X}_{per}^{B}\left( 
\mathbb{R}
_{y}^{N};\mathcal C_b\right)$, of the density of $\mathcal{C}_{per}\left( Y\times
Z\right) $ in $\mathfrak{X}_{per}^{B}\left( 
\mathbb{R}
_{y}^{N};\mathcal C_b\right) ,$ of the continuity of $M$ on $%
\mathfrak{X}_{per}^{B}\left( 
\mathbb{R}
_{y}^{N};\mathcal C_b\right) $ and of the continuity of $v\rightarrow
v^{\varepsilon }$ from $\mathfrak{X}_{per}^{B}\left( 
\mathbb{R}
_{y}^{N};\mathcal C_b\right) $\ to $L^{B}\left( \Omega \right) $, (see \eqref{tracebounds}).
\end{proof}

\medskip 
Now we endow $\mathfrak{X}_{per}^{B}\left(\mathbb R_y^N;\mathcal C_b\right) $ with another norm. Indeed we define $\mathfrak{X}%
	_{per}^{B}\left( 
	\mathbb{R}
	_{y}^{N}\times 
	\mathbb{R}
	_{z}^{N}\right) $ the closure of $\mathcal{C}_{per}\left( Y\times Z\right) $
	in $L_{loc}^{B}\left( 
	\mathbb{R}
	_{y}^{N}\times 
	\mathbb{R}
	_{z}^{N}\right) $ with the norm $$\left\Vert u\right\Vert _{\Xi^{B}}:= \sup_{0<\varepsilon \leq 1}\left\|u\left(\frac{x}{\varepsilon}, \frac{y}{\varepsilon}\right)\right\|_{B, 2 B_N}.$$ 

Via Riemann-Lebesgue lemma the above norm is equivalent to 
	$\|u\|_{L^B(Y\times Z)},
	$ thus in the sequel we will consider this one. 
	
For the sake of completeness, we state the following result which proves that the latter norm is controlled by the one defined in \eqref{originalnormLBPer}, thus together with Lemma \ref{lemma2.2}, it provides the eqivalence among the introduced norms in $\mathfrak{X}_{per}^B(\mathbb R^N_y;\mathcal C_b)$.
The proof is postponed in the Appendix.

\begin{proposition}\label{prop2.1}
	It results that
	$\mathfrak{X}_{per}^{B}\left(
	\mathbb R_y^N;\mathcal C_b\right) \subset L_{per}^{B}\left( Y\times Z\right) =%
	\mathfrak{X}_{per}^{B}\left( 
	\mathbb R_y^N\times 
	\mathbb R_z^N\right) $ and $\left\Vert u\right\Vert _{B,Y \times Z }\leq c\left\Vert
	u\right\Vert _{\Xi ^{B}\left(
		\mathbb R_y^N;\mathcal C_b\left( 
		\mathbb R_z^N\right) \right) }$ for all $u\in \mathfrak{X}_{per}^{B}\left( 
	\mathbb R_y^N;\mathcal C_b\right) .$
\end{proposition}

\color{black}

\subsection{Reiterated two-scale convergence in Orlicz spaces}

\bigskip Generalizing definitions in \cite{fotso nnang 2012, nnang reit,
Nnang Ph. D. These}  we introduce
\begin{align*}
L_{per}^{B}\left( \Omega \times Y\times Z\right) =\Big\{ u\in
L_{loc}^{B}\left( \Omega \times 
\mathbb{R}
_{y}^{N}\times \mathbb{R}
_{z}^{N}\right) :\text{for }a.e\text{ }x\in \Omega ,u\left( x,\cdot,\cdot\right) \in
L_{per}^{B}\left( Y\times Z\right)  \\ 
\left. \text{ and }\iiint_{\Omega \times Z}B\left( \left\vert u\left(
x,y,z\right) \right\vert \right) dxdydz<\infty \right\}.
\end{align*}

\noindent We are in position to define reiterated two-scale convergence:

\begin{definition}\label{def3s}
A sequence of functions $\left( u_{\varepsilon }\right) _{\varepsilon}  \subseteq L^{B}\left( \Omega \right) $ is said to be:
\begin{itemize}
\item[-]weakly reiteratively two-scale convergent in $L^{B}\left( \Omega \right) $
to a function $u_{0}\in L_{per}^{B}\left( \Omega \times Y\times
Z\right) $ if 
\begin{equation}
\label{2}
\int_{\Omega }u_{\varepsilon }f^{\varepsilon }dx\rightarrow 
\iiint_{\Omega \times Y\times Z}u_{0}fdxdydz, \hbox{ for all } f\in L^{%
\widetilde{B}}\left( \Omega ;\mathcal{C}_{per}\left( Y\times Z\right) \right),
\end{equation}%
as $\varepsilon \to 0$,
\item[-]strongly reiteratively two-scale convergent in $L^{B}\left( \Omega \right) $\
to $u_{0}\in L_{per}^{B}\left( \Omega \times Y\times Z\right) $\ if for $%
\eta >0$ and $f\in L^{B}\left( \Omega ;\mathcal{C}_{per}\left( Y\times
Z\right) \right) $ verifying $\left\Vert u_{0}-f\right\Vert _{B,\Omega
\times Y\times Z}\leq \frac{\eta }{2}$ there exists $\rho >0$ such that $%
\left\Vert u_{\varepsilon }-f^{\varepsilon }\right\Vert _{B,\Omega }\leq
\eta $ for all $0<\varepsilon \leq \rho.$
\end{itemize}
\end{definition}

When (\ref{2}) happens 
we denote it by "$%
u_{\varepsilon }\rightharpoonup u_{0}$ in $L^{B}\left( \Omega
\right)-$\ weakly reiteratively two-scale " 
and we will say that 
$u_{0}$ is the weak reiterated two-scale limit in $L^{B}\left( \Omega
\right) $ of the sequence $\left( u_{\varepsilon }\right) _{\varepsilon}.$

\begin{remark}
	The above definition extends in a canonical way, arguing in components, to vector valued functions.
	\end{remark}

\begin{lemma}
If $u\in L^{B}\left( \Omega ;\mathcal{C}_{per}\left( Y\times Z\right)
\right) $ then $u^{\varepsilon }\rightharpoonup  $u in $L^{B}\left(
\Omega \right) $ weakly reiteratively two-scale, and we have $\underset{\varepsilon
\rightarrow 0}{\lim }\left\Vert u^{\varepsilon }\right\Vert _{B,\Omega
}=\left\Vert u\right\Vert _{B,\Omega \times Y\times Z}$
\end{lemma}

\begin{proof}
Let $u\in L^{B}\left( \Omega ;\mathcal{C}_{per}\left( Y\times Z\right)
\right) $ and $f\in L^{\widetilde{B}}\left( \Omega ;\mathcal{C}_{per}\left(
Y\times Z\right) \right) $ then $uf\in L^{1}\left( \Omega ;\mathcal{C}%
_{per}\left( Y\times Z\right) \right) $ and  $$\lim_{\varepsilon
\rightarrow 0} \int_{\Omega }u^{\varepsilon }f^{\varepsilon
}dx=\iiint_{\Omega \times Y\times Z}ufdxdydz.$$ Similary $%
\hbox{ for all } \delta >0,B\left( \left|\frac{u}{\delta }\right|\right) \in L^{1}\left( \Omega ;%
\mathcal{C}_{per}\left( Y\times Z\right) \right) $ and the result follows.
\end{proof}


We are in position of proving a first sequential compactness result.

\begin{proposition}\label{propcomp}
Given a bounded sequence $\left(
u_{\varepsilon }\right) _{\varepsilon}\subset L^{B}\left( \Omega
\right) ,$ one can extract a not relabelled subsequence such that $%
\left( u_{\varepsilon }\right) _{\varepsilon}$ is
weakly reiteratively two-scale convergent in $L^{B}\left( \Omega \right) .$
\end{proposition}

\begin{proof}
For $\varepsilon >0,$ set $L_{\varepsilon }\left( \psi \right)
=\int_{\Omega }u_{\varepsilon }\left( x\right) \psi \left( x,\frac{x}{%
\varepsilon },\frac{x}{\varepsilon ^{2}}\right) dx,\psi \in L^{\widetilde{B}%
}\left( \Omega ;\mathcal{C}_{per}\left( Y\times Z\right) \right) .$  Clearly $%
L_{\varepsilon }$ is a linear form and we have \begin{equation}\left\vert L_{\varepsilon
}\left( \psi \right) \right\vert \leq 2\left\Vert u_{\varepsilon
}\right\Vert _{B,\Omega }\left\Vert \psi ^{\varepsilon }\right\Vert _{%
\widetilde{B},\Omega }\leq c\left\Vert \psi \right\Vert _{L^{\widetilde{B}%
}\left( \Omega ;\mathcal{C}_{per}\left( Y\times Z\right) \right) },\label{Flinear}
\end{equation}
for a constant
$c$ independent on $\varepsilon $\ and $\psi .$ Thus $\left(
L_{\varepsilon }\right) _{\varepsilon}$\ is bounded in $\left[ L^{%
\widetilde{B}}\left( \Omega ;\mathcal{C}_{per}\left( Y\times Z\right)
\right) \right] ^{\prime }$. Since $L^{\widetilde{B}}\left( \Omega ;\mathcal{C%
}_{per}\left( Y\times Z\right) \right) $ is a separable Banach space, we can
extract a not relabelled subsequence, such that, as  $\varepsilon \rightarrow 0,$%
\begin{equation*}
L_{\varepsilon }\rightarrow L_{0},\text{ in }\left[ L^{\widetilde{B}}\left(
\Omega ;\mathcal{C}_{per}\left( Y\times Z\right) \right) \right] ^{\prime }%
\text{ weakly}\ast .
\end{equation*}%
In order to characterize $L_{0}$ note that \eqref{Flinear} ensures
$$\left\vert L_{0}\left(
\psi \right) \right\vert \leq c\left\Vert \psi \right\Vert _{\widetilde{B}%
,\Omega \times Y\times Z} \hbox{ for every }\psi \in L^{\widetilde{B}}\left( \Omega ;\mathcal{C%
}_{per}\left( Y\times Z\right) \right) .
$$ 
Recalling that $L^{\widetilde{B}}\left(
\Omega ;\mathcal{C}_{per}\left( Y\times Z\right) \right) $\ is dense in $%
L_{per}^{\widetilde{B}}\left( \Omega \times Y\times Z\right),$  $L_{0}$ 
can be extended by continuity to an element of $\left[ L_{per}^{\widetilde{B}%
}\left( \Omega \times Y\times Z\right) \right] ^{\prime }\overline{=}%
L_{per}^{B}\left( \Omega \times Y\times Z\right) $. Thus there exist $%
u_{0}\in L_{per}^{B}\left( \Omega \times Y\times Z\right) $ such that 
\begin{equation*}
\lim_{\varepsilon \to 0}\int_{\Omega }u_{\varepsilon }\left( x\right) \psi \left( x,\frac{x}{%
\varepsilon },\frac{x}{\varepsilon ^{2}}\right) dx =
\iiint_{\Omega \times Y\times Z}u_0\left( x,y,z\right) \psi \left(
x,y,z\right) dxdydz, 
\end{equation*}
for all $\psi \in L^{\widetilde{B}}\left( \Omega ;\mathcal{C}%
_{per}\left( Y\times Z\right) \right).$
\end{proof}

The proof of the following results are omitted, since they are consequence of 'standard' density results and are very similar to the (non reiterated) two-scale case (see for instance \cite{fotso nnang 2012}).

\begin{proposition}
If a sequence $\left( u_{\varepsilon }\right) _\varepsilon$ is weakly
reiteratively two-scale convergent in $L^{B}\left( \Omega \right) $
to $u_0\in L_{per}^{B}\left( \Omega \times Y\times Z\right) $ then
\begin{itemize}
\item[(i)]  $u_{\varepsilon }\rightharpoonup\int_{Z}u_0\left(
\cdot,\cdot,z\right) dz$ in $L^{B}\left( \Omega \right) $ weakly two-scale,
and
\item[(ii)]  $u_{\varepsilon }\rightharpoonup \widetilde{u_0}$ in $L^{B}\left(
\Omega \right) $-weakly as $\varepsilon \to 0$ where $\widetilde{u_0}
\left( x\right) =\iint_{Y\times Z}u_0\left( x,\cdot,\cdot\right) dydz.$
\end{itemize}
\end{proposition}

\begin{proposition}\label{prop2.4}
Let	$\mathfrak{X}_{per}^{B,\infty}\left( 
\mathbb R_y^N;\mathcal C_b\right):=\mathfrak{X}_{per}^B\left( 
\mathbb R_y^N;\mathcal C_b\right)\cap L^\infty(\mathbb R^N_y\times \mathbb R^N_z).$
If a sequence $\left( u_{\varepsilon }\right) _{\varepsilon}$ is
weakly reiteratively two-scale convergent in $L^{B}\left( \Omega \right) $
to $u_{0}\in L_{per}^{B}\left( \Omega \times Y\times Z\right) $ we also have 
$\int_{\Omega }u_{\varepsilon }f^{\varepsilon }dx\rightarrow
\iiint_{\Omega \times Y\times Z}u_{0}fdxdydz,$ for all $f\in \mathcal{C}\left( 
\overline{\Omega }\right) \otimes \mathfrak{X}_{per}^{B,\infty}\left( 
\mathbb R_y^N;\mathcal C_b\right) .$
\end{proposition}

\begin{corollary}\label{corollary2.1}
Let $v\in \mathcal{C}\left( \overline{\Omega };\mathfrak{X}_{per}^{B,\infty
}(\mathbb R_y^N;\mathcal C_b) \right).$ Then $v^{\varepsilon } \rightharpoonup v$ in $L^{B}\left( \Omega \right) $- weakly reiteratively two-scale as $\varepsilon
\rightarrow 0.$
\end{corollary}

\begin{remark}
	\begin{itemize}
\item[(1)] If $v\in L^{B}\left( \Omega ;\mathcal{C}_{per}\left( Y\times Z\right)
\right) ,$ then $v^{\varepsilon }\rightarrow v$ in $L^{B}\left( \Omega
\right) $-strongly reiteratively two-scale as $\varepsilon \rightarrow 0.$

\item[(2)] If $\left( u_{\varepsilon }\right) _{\varepsilon}\subset
L^{B}\left( \Omega \right) $ is strongly reiteratively two-scale convergent in $
L^{B}\left( \Omega \right) $\ to $u_{0}\in L_{per}^{B}\left( \Omega \times
Y\times Z\right)$ then 
\begin{itemize}
\item[(i)]  $u_{\varepsilon }\rightharpoonup u_{0}$ in $L^{B}\left(
\Omega \right) $ weakly reiteratively two-scale as $\varepsilon \rightarrow 0;$
\item[(ii)] $\left\Vert u_{\varepsilon }\right\Vert _{B,\Omega }\rightarrow
\left\Vert u_{0}\right\Vert _{B,\Omega \times Y\times Z}$ as $
\varepsilon \rightarrow 0.$
\end{itemize}
\end{itemize}
\end{remark}

The following result is crucial to provide a notion of weakly reiterated two-scale convergence in Orlicz-Sobolev spaces and for the sequential compactness result on $
W^{1}L^{B}\left( \Omega \right).$ It extends and presents an alternative proof of \cite[Theorem 4.1]{fotso nnang 2012}.

To this end, recall first that $L_{per}^1\left(
Y;W_{\#}^{1}L^{B}\left( Z\right) \right)$ denotes the space of functions $u \in L^1_{per}(Y\times Z)$, such that $u(y,\cdot) \in W_{\#}^{1}L^{B}\left( Z\right) $, for a.e. $y \in Y$. 

\begin{proposition}\label{mainprop3s}
Let $\Omega $ be a bounded open set in $
\mathbb R_x^N$, and $\left( u_{\varepsilon
}\right) _{\varepsilon}$ bounded in $W^{1}L^{B}\left( \Omega \right).$ 
There exist a not relabelled subsequence, $u_{0}\in W^{1}L^{B}\left(
\Omega \right),$ $\left(u_{1},u_{2}\right) \in L^{1}\left(
\Omega ;W_{\#}^{1}L^{B}\left( Y\right) \right) \times L^{1}\left(
\Omega ;L^1_{per}\left( Y;W_{\#}^{1}L^{B}\left( Z\right) \right)
\right) $ such that:
\begin{itemize}
\item[(i)] $u^{\varepsilon }\rightharpoonup u_{0}$ weakly reiteratively two-scale in $L^{B}\left(
\Omega \right) $,
\item[(ii)] $D_{x_{i}}u^{\varepsilon }\rightharpoonup
D_{x_{i}}u_{0}+D_{y_{i}}u_{1}+D_{z_{i}}u_{2}$ weakly reiteratively two-scale in $L^{B}\left( \Omega\\
\right) $, $1\leq i\leq N$,
\end{itemize}
as $\varepsilon \to 0$.
\end{proposition}

\begin{corollary}
	If $\left( u_{\varepsilon }\right) _{\varepsilon}$ is such that$\ \
	u_{\varepsilon }\rightharpoonup v_{0}$ weakly reiteratively two-scale in $W^{1}L^{B}\left( \Omega \right)$,  we have:
	\begin{itemize}
		\item[(i)] 	$u_{\varepsilon }\rightharpoonup
		\int_{Z}v_0\left( \cdot ,\cdot ,z\right) dz$ weakly two-scale in $W^{1}L^{B}\left( \Omega \right) $, 
		\item[(ii)] $u_{\varepsilon }\rightharpoonup 
		\widetilde{v_0}$ in $W^{1}L^{B}\left( \Omega \right) $-weakly, where $\widetilde{v_0}\left( x\right) =
		\iint_{Y\times Z}v_0\left( x,\cdot ,\cdot\right) dydz.$ 
	\end{itemize}  
\end{corollary}

\begin{proof}[Proof of Proposition \ref{mainprop3s}]
We recall that : $L^{B}\left( \Omega _{1}\times \Omega _{2}\right)
\subset L^{1}\left( \Omega _{1};L^{B}\left( \Omega _{2}\right) \right) .$
Moreover since $B$ satisfies $\triangle _{2},$ there exist $q>p>1$ such that: $%
L^{q}\left( \Omega \right) \hookrightarrow L^{B}\left( \Omega \right)
\hookrightarrow L^{p}\left( \Omega \right)$,
(relying on \cite[Proposition 2.4]{DG} (see also \cite[Proposition 3.5]{C} ) and a standard argument based on decreasing rearrangements), where the arrows stand for continuous embedding.

Let $\left( u_{\varepsilon }\right)_\varepsilon $ be bounded in $L^{B}\left( \Omega
\right) .$ Then it is bounded in $L^{p}\left( \Omega \right) $ and we have:

\begin{itemize}
	\item[(i)] $u_{\varepsilon }\rightharpoonup U_{0}$ weakly reiteratively two-scale in $L^{B}\left(
\Omega \right)$,
\item[(ii)] $u_{\varepsilon }\rightharpoonup u_{0}$ in $W^{1}L^{B}\left( \Omega \right) $,
\item[(i)'] $u_{\varepsilon }\rightharpoonup 
U_{0}^{\prime }$ weakly reiteratively two-scale in $L^{p}\left( \Omega \right),
$
\item[(ii)'] $u_{\varepsilon }\rightharpoonup u_{0}^{^{\prime }}$ in $%
W^{1,p}\left( \Omega \right) $.
\end{itemize}

By classical results (see for instance \cite{All2} and \cite{Elvira 1}), we know that
$$u_0'=U'_0,$$
on the other hand, using $W^{1,p}\left( \Omega \right) $-weak$\hookrightarrow 
\mathcal{D}^{\prime }\left( \Omega \right) -$weak  and $%
W^{1}L^{B}\left( \Omega \right) $-weak$\hookrightarrow \mathcal{D}^{\prime
}\left( \Omega \right) -$weak, we  deduce that $u_{0}^{\prime }=u_{0}\in
W^{1}L^{B}\left( \Omega \right). $
Moreover, since  $L^{p'}(\Omega)\hookrightarrow L^{\tilde B}(\Omega)$, it results then  $L^{p^{\prime }}\left( \Omega ;\mathcal{C}_{per}\left(
Y\times Z\right) \right) $ $\subset L^{\widetilde{B}}\left( \Omega ;\mathcal{%
	C}_{per}\left( Y\times Z\right) \right) $, thus
$$
U_0=U_0',
$$
thus
$$
U_0=U_0'=u_0=u_0'.
$$

We also have
\begin{itemize}
\item[(iii)] $D_{x_{i}}u_{\varepsilon }\rightharpoonup \tilde w$ weakly reiteratively two-scale in $%
L^{B}\left( \Omega \right) $, $1\leq i\leq N$,
\item[(iii)'] $D_{x_{i}}u_{\varepsilon }\rightharpoonup
D_{x_{i}}u_{0}+D_{y_{i}}u_{1}+D_{z_{i}}u_{2}$ weakly reiteratively two-scale in $L^{p}\left( \Omega
\right) $, $1\leq i\leq N$, with $\left( u_{1},u_{2}\right) \in L^{p}_{per}\left( \Omega
;W_{\#}^{1,p}\left( Y\right) \right) \times L^{p}\left( \Omega
;L^{p}_{per}\left( Y;W_{\#}^{1,p}\left( Z\right) \right) \right)$ (see \cite{All2} and \cite{Elvira 1}).
\end{itemize}
Arguing in components, as done above, we are lead to conclude that

$$\tilde w= D_{x_{i}}u_{0}+D_{y_{i}}u_{1}+D_{z_{i}}u_{2}\in L_{per}^{B}\left( \Omega
\times Y\times Z\right) $$ and $D_{x_{i}}u_{0}\in L^{B}\left( \Omega \right)
\subset L_{per}^{B}\left( \Omega \times Y\times Z\right),$ as $u_{0}\in
W^{1}L^{B}\left( \Omega \right) .$ Therefore $%
\tilde w-D_{x_{i}}u_{0}=D_{y_{i}}u_{1}+D_{z_{i}}u_{2}\in L_{per}^{B}\left( \Omega
\times Y\times Z\right) $. 
By Jensen's inequality, $B\left(
\int_{Z}\left\vert \tilde w\right\vert dz\right) \leq \left( \int_{Z}B\left(
\left\vert \tilde w\right\vert \right) dz\right) $ then
$$\iint_{\Omega \times Y}B\left( \int_{Z}\left\vert \tilde w\right\vert dz\right)
dxdy\leq \iint_{\Omega \times Y}\int_{Z}B\left( \left\vert \tilde w \right\vert
\right) dzdxdy<\infty.
$$

 Since $B$ satisfies $ \triangle _{2}$, $
\int_{Z}\tilde w dz=D_{x_{i}}u_{0}+D_{y_{i}}u_{1}\in L_{per}^{B}\left( \Omega \times
Y\right) $ with $D_{x_{i}}u_{0}\in L^{B}\left( \Omega \right) \subset
L_{per}^{B}\left( \Omega \times Y\right)$.
Therefore $%
\int_Z \tilde w dz-D_{x_{i}}u_{0}=D_{y_{i}}u_{1}\in L_{per}^{B}\left( \Omega \times Y\right)
\subset L^{1}\left( \Omega ;L^{B}_{per}\left( Y\right) \right) $.
On the other hand
$u_{1}\in L^{p}_{per}\left( \Omega
;W_{\#}^{1,p}\left( Y\right) \right) $, i.e.  for almost all $%
x,u_{1}\left( x,\cdot\right) \in W_{\#}^{1,p}\left( Y\right) $ $%
=\left\{ v\in W^{1,p}_{per}\left( Y\right) :\int_{Y}vdy=0\right\} $ and 
$D_{y_{i}}u_{1}\left( x,\cdot\right) \in L^{B}_{per}\left( Y\right)$. 
In particular $u_{1}\left( x,\cdot\right) \in L_{per}^{p}\left( Y\right)
\subset L_{per}^{1}\left( Y\right) $.

\noindent To complete the proof it remains to show that every $v\in L^{p}\left( Y\right)$ with $D_{y_{i}}v\in L^{B}_{per}\left( Y\right)$ 
is in $L^{B}_{per}\left( Y\right) .$

Set $u=u-M\left( u\right) +M\left( u\right),$ where $M$ is the averaging operator in \eqref{M}. Then, by Poincar\'e inequality, it results%
\[
\begin{tabular}{l}
$\left\Vert u\right\Vert _{B,Y}\leq \left\Vert u-M\left( u\right)
\right\Vert _{B,Y}+\left\Vert M\left( u\right) \right\Vert _{B,Y}\leq
c\left\Vert Du\right\Vert _{B,Y}+\left\Vert M\left( u\right) \right\Vert
_{B,Y}\leq $ \\ 
$c\left\Vert Du\right\Vert _{B,Y}+c_{1}\left( 1+\left\Vert u\right\Vert
_{L^{1}\left( Y\right) }\right) <\infty .$%
\end{tabular}%
\]%
The last inequality being consequence of the fact that $\underset{t\rightarrow 0}{\lim }B\left( t\right) =0,\exists
c_{1}>0,B\left( \frac{1}{c_{1}}\right) <1.$  Hence,

\noindent $\int_{Y}B\left( \frac{%
	\left\vert M\left( u\right) \right\vert }{\left( 1+\left\vert M\left(
	u\right) \right\vert \right) c_{1}}\right) dy\leq $ $\int_{Y}B\left( \frac{1%
}{c_{1}}\right) dy\leq 1;$ that is $\left\Vert M\left( u\right) \right\Vert
_{B,Y}\leq \left( 1+\left\vert M\left( u\right) \right\vert \right)
c_{1}=\left( 1+\left\vert \int_{Y}udy\right\vert \right) c_{1}\leq
c_{1}\left( 1+\left\Vert u\right\Vert _{L^{1}\left( Y\right) }\right).$

Thus we can conclude that $u_{1}\in L^{1}_{per}\left( \Omega ;W_{\#}^{1}L^{B}\left(
Y\right) \right) .$

For what concerns $u_2$ we can argue in a similar way. Recall that
\begin{align*}
\tilde w =D_{x_{i}}u_{0}+D_{y_{i}}u_{1}+D_{z_{i}}u_{2}\in
L_{per}^{B}\left( \Omega \times Y\times Z\right) ,D_{x_{i}}u_{0}\in
L^{B}\left( \Omega \right) , \\ 
u_{1}\in L^{1}\left( \Omega ;W_{\#}^{1}L^{B}\left( Y\right) \right)
,u_{2}\in L^{p}\left( \Omega ;L^{p}_{per}\left( Y;W_{\#}^{1,p}\left(
Z\right) \right) \right).
\end{align*}
So  $D_{z_{i}}u_{2}=\tilde w -\left(
D_{x_{i}}u_{0}+D_{y_{i}}u_{1}\right) \in L_{per}^{B}\left( \Omega \times
Y\times Z\right) \subset L^{1}\left( \Omega ;L_{per}^{1}\left(
Y;L^{B}\left( Z\right) \right) \right)$, thus $D_{z_{i}}u_{2}\left(
x,y,\cdot\right) \in L_{per}^{B}\left( Z\right) $ for almost all $\left(
x,y\right) \in \Omega \times 
\mathbb{R}
_{y}^{N};\int_{Z}u_{2}\left( x,y,\cdot\right) dz=0$ as \ $u_{2}\left(
x,y,\cdot\right) \in W_{\#}^{1,p}\left( Z\right)$. Consequently, since $u_{2}\left(
x,y,\cdot\right) \in L_{per}^{p}\left( Z\right) \subset L_{per}^{1}\left(
Z\right) ,D_{z_{i}}u_{2}\left( x,y,\cdot\right) \in L_{per}^{B}\left( Z\right)$,  exploiting Poincare' inequality with the averaging operator $M$, as done above, it results that
$u_{2}\left( x,y,\cdot\right) \in W_{\#}^{1}L^{B}\left( Z\right).$

\noindent Since
$L^{p}\left( \Omega ;L^{p}_{per}\left( Y;W_{\# }^{1,p}\left( Z\right)
\right) \right) =$ $L^{p}_{per}\left( \Omega \times Y;W_{\#}^{1,p}\left(
Z\right) \right) \subset $ \hfill 

\noindent $L^{1}_{per}\left( \Omega \times Y;W_{\#}^{1,p}\left(
Z\right) \right) =L^{1}\left( \Omega ;L^{1}_{per}\left(
Y;W_{\#}^{1,p}\left( Z\right) \right) \right) $, we deduce that $%
u_{2}$ $\in L^{1}_{per}\left( \Omega ;L^{1}\left( Y;W_{\#}^{1}L^{B}\left(
Z\right) \right) \right) .$ 
\end{proof}

In view of the next applications, we underline that, under the assumptions of the above proposition, the canonical injection $%
	W^{1}L^{B}\left( \Omega \right) \hookrightarrow L^{B}\left( \Omega \right) $
	is compact. 

\section{Homogenization of integral energies with
convex and non standard growth}

In this section we study the asymptotic behaviour of \eqref{PR1} under the assumptions $(H_1)-(H_4)$, stated above.
  We start by recalling the properties satisfied by $F_\varepsilon$ in \eqref{Fepsilon}.

Since the function $f$ in \eqref{Fepsilon} is convex in the last argument and satisfies $(H_4)$, it results that (cf. \cite{fotso nnang 2012}) there exists a constant $c>0$ such that: \begin{equation}
\label{Blip}
\left\vert
f\left( y,z,\lambda \right) -f\left( y,z,\mu \right) \right\vert \leq c\frac{%
1+B\left( 2\left( 1+\left\vert \lambda \right\vert +\left\vert \mu
\right\vert \right) \right) }{1+\left\vert \lambda \right\vert +\left\vert
\mu \right\vert }\left\vert \lambda -\mu \right\vert 
\end{equation} for all $\lambda ,\mu
\in 
\mathbb{R}
^{nN}$ and for a.e. $y\in 
\mathbb{R}
^{N}_y$ and for all $z \in \mathbb R^N_z$.
 Hence for fixed $%
\varepsilon >0$ and for $v\in $ $W_{0}^{1}L^{B}\left( \Omega ;%
\mathbb{R}
^{nN}\right) ,$ the function $x\mapsto f\left( \frac{x}{\varepsilon },\frac{x%
}{\varepsilon ^{2}},v\left( x\right) \right) $ from $\Omega $ into $%
\mathbb{R}
_{+}$ denoted by $f^{\varepsilon }\left( \cdot,\cdot,v\right) $,\ is well defined as an element of $L^{1}\left( \Omega \right) $ and it results (arguing as in \cite[Proposition 3.1]{fotso nnang 2012})
\begin{align}\label{forestimate5.1}
\left\Vert f^{\varepsilon }\left(\cdot ,\cdot,v\right) -f^{\varepsilon }\left(\cdot,\cdot,w\right) \right\Vert _{L^{1}\left( \Omega \right) }\leq \\ \nonumber
c\left(
\left\Vert 1\right\Vert _{\widetilde{B},\Omega }+\left\Vert b\left(
1+\left\vert v\right\vert +\left\vert w\right\vert \right) \right\Vert _{%
\widetilde{B},\Omega }\right) \left\Vert v-w\right\Vert _{\left( L^{B}\left(
\Omega \right) \right) ^{nN}}.
\end{align}%
Moreover, $(H_4)$ ensures that for $v\in W_{0}^{1}L^{B}\left( \Omega ;\mathbb R^n\right) $ such that $\left\Vert Dv\right\Vert _{\left( L^{B}\left(
\Omega \right) \right) ^{nN}}\geq 1$, we have 
$$c_{1}\left\Vert Dv\right\Vert
_{\left( L^{B}\left( \Omega \right) \right) ^{nN}}\leq \left\Vert
f^{\varepsilon }\left( \cdot,\cdot,Dv\right) \right\Vert _{L^{1}\left( \Omega
\right) }\leq c_{2}\left( 1+\left\Vert Dv\right\Vert _{\left( L^{B}\left(
\Omega \right) \right) ^{nN}}\right). $$  Consequently it results that $F_{\varepsilon }$ is
continuous, strictly convex and coercive thus there exists a unique $%
u_{\varepsilon }\in W_{0}^{1}L^{B}(\Omega)$ solution of the minimization problem $\underset{v\in W_{0}^{1}L^{B}(\Omega) }{\min }F_{\varepsilon }\left( v\right) $, i.e. $$F_{\varepsilon }\left(
u_{\varepsilon }\right) =\underset{v\in W_{0}^{1}L^{B}(\Omega) }{\min }F_{\varepsilon }\left( v\right).$$

\noindent Let $\psi \in \mathcal{C}\left( \overline{\Omega };\mathcal{C}_{per}\left( Y\times Z\right) \right)^N.$ For fixed $x\in \overline{%
\Omega}$ the function $\left( y,z\right)\in 
\mathbb{R}
_{y}^{N}\times 
\mathbb{R}
_{z}^{N} \mapsto f\left( y,z,\psi \left(
x,y,z\right) \right) \in 
\mathbb{R}
_{+}$ denoted by $f\left( \cdot,\cdot,\psi \left( x,\cdot,\cdot\right) \right) $ lies in $
L^{\infty }\left( 
\mathbb{R}
_{y}^{N};\mathcal C_b\left( \mathbb{R}
_{z}^{N}\right) \right) .$ Hence one can define the function $x\in \overline \Omega \mapsto
f\left(\cdot,\cdot,\psi \left( x,\cdot,\cdot\right) \right)$ 
and denote it by $f\left( \cdot,\cdot,\psi \right) $) as
element of $\mathcal{C}\left( \overline{\Omega };L^{\infty }\left( 
\mathbb{R}
_{y}^{N};\mathcal C_b\left( 
\mathbb{R}
_{z}^{N}\right) \right) \right).$  

Therefore, for fixed $\varepsilon >0$,
the function $x\mapsto f\left( \frac{x}{\varepsilon },\frac{x}{%
\varepsilon ^{2}},\psi \left( x,\frac{x}{\varepsilon },\frac{x}{\varepsilon
^{2}}\right) \right) $ denoted by $f^{\varepsilon }\left( \cdot,\cdot,\psi
^{\varepsilon }\right) $ is an element of $L^{\infty }\left( \Omega
\right) $. Moreover, in view of the periodicity of $f\left( \cdot,\cdot,\psi \right)$, which is
in $\mathcal{C}\left( \overline{\Omega };L_{per}^{\infty }\left( Y;\mathcal{C%
}_{per}^{\infty }\left( Z\right) \right) \right) $\ for all $\psi \in 
\mathcal{C}\left( \overline{\Omega };\mathcal{C}_{per}\left( Y\times
Z\right) \right)^{N},$ the following result holds:

\begin{proposition}\label{proposition3.1}
For every $v\in \mathcal{C}\left( \overline{\Omega };\mathcal{C}_{per}\left(
Y\times Z\right) \right) ^{N}$ one has 
\begin{equation*}
\underset{\varepsilon \rightarrow 0}{\lim }\int_{\Omega }f\left(\frac{x}{%
\varepsilon },\frac{x}{\varepsilon ^{2}},v\left( x,\frac{x}{\varepsilon },%
\frac{x}{\varepsilon ^{2}}\right) \right) dx=\iiint_{\Omega \times Y\times
Z}f\left(y,z,v\left( x,y,z\right) \right) dxdydz.
\end{equation*}
Futhermore, the mapping $v\in \mathcal{C}\left( \overline{\Omega };\mathcal{C}_{per}\left( Y\times
Z\right) \right) ^{N} \mapsto $ $f\left( \cdot,\cdot,v\right) \in L_{per}^{1}\left( \Omega \times Y\times
Z\right)$ extends by continuity to a mapping still denoted by $%
v\mapsto $ $f\left( \cdot,\cdot,v\right) $ from $\ \left( L_{per}^{B}\left(
\Omega \times Y\times Z\right) \right) ^{N}$\ into $
L_{per}^{1}\left( \Omega \times Y\times Z\right)$\ such that:%
\begin{align}\label{estimate5.1}
\left\Vert f\left( \cdot,\cdot,v\right) -f\left(
\cdot,\cdot,w\right) \right\Vert _{L^{1}(\Omega \times Y \times Z)}\leq  \\ \nonumber 
c\left( \left\Vert 1\right\Vert _{\widetilde{B},\Omega }+\left\Vert b\left(
1+\left\vert v\right\vert +\left\vert w\right\vert \right) \right\Vert _{%
\widetilde{B},\Omega \times Y\times Z}\right) \left\Vert v-w\right\Vert
_{\left( L_{per}^{B}\left( \Omega \times Y\times Z\right) \right) ^{N}}
\end{align}%
\ for all $v,w\in \left( L_{per}^{B}\left( \Omega \times Y\times Z\right)
\right)^{N}$\ .
\end{proposition}

\begin{proof}
It is a simple adaptations of the proof of \cite[Proposition 5.1]{fotso nnang 2012}, relying in turn on Corollary \ref{corollary2.1}. Moreover \eqref{estimate5.1} follows by \eqref{Blip} and by arguments identical to those used to deduce \eqref{forestimate5.1}, and omitted here since already presented in \cite[Proposition 3.1]{fotso nnang 2012}, which in turn require the application of Lemma \ref{lemma2.2}  
\end{proof}

\begin{corollary}\label{corollary3.1}
Let $\phi _{\varepsilon }\left( x\right) :=\psi _{0}+\varepsilon \psi
_{1}\left( x,\frac{x}{\varepsilon }\right) +\varepsilon ^{2}\psi _{1}\left(
x,\frac{x}{\varepsilon },\frac{x}{\varepsilon ^{2}}\right) $ for $x\in
\Omega ,$ where $\psi _{0}\in \mathcal{C}_{0}^{\infty }\left( \Omega \right),\psi _{1}\in \left[ \mathcal{C}_{0}^{\infty }\left( \Omega \right)
\otimes \mathcal{C}_{per}^{\infty }\left( Y\right) \right]$and $\psi
_{2}\in \left[ \mathcal{C}_{0}^{\infty }\left( \Omega \right) \otimes 
\mathcal{C}_{per}^{\infty }\left( Y\right) \otimes \mathcal{C}_{per}^{\infty
}\left( Z\right) \right]$, then, as $\varepsilon \rightarrow 0,$%
\begin{align*}
\underset{\varepsilon \rightarrow 0}{\lim }\int_{\Omega }f\left(\frac{x}{%
\varepsilon },\frac{x}{\varepsilon ^{2}},D\phi _{\varepsilon }\right) dx=
\\ 
\iiint_{\Omega \times Y\times Z}f\left(y,z,D\psi _{0}+D_{y}\psi
_{1}+D_{z}\psi _{2}\right) dxdydz.
\end{align*}
\end{corollary}

\begin{proof}
It is a simple adaptations of \cite[Corollary 5.1]{fotso nnang 2012}, relying on \eqref{Blip} and \eqref{forestimate5.1}, observing that 
$f^\varepsilon(\cdot, \cdot, (D\psi_0 + D_y \psi_1+ D_z \psi_2)^\varepsilon) \in C(\overline \Omega; \mathfrak{X}_{per}^{B,\infty
}(\mathbb R_y^N;\mathcal C_b))$ and Corollary \ref{corollary2.1} applies.
\end{proof}

\smallskip

Now, we observe that, thanks to the density of $\mathcal{D}(
\Omega) $ in $W_{0}^{1}L^{B}(\Omega),$ of $\mathcal{C}_{per}^{\infty }(Y)/
\mathbb{R}
$ in $W_{\#}^{1}L_{per}^{B}(Y) $ and that of $\mathcal{C}_{per}^{\infty }(Y) \otimes \mathcal{C}_{per}^{\infty }( Z) /\mathbb R$ in $L_{per}^1( Y;W_{\#}^{1}L^{B}(Z)),$ the space 
\begin{equation}\label{Finfty0}
F_{0}^{\infty }:=\mathcal{D}(\Omega) \times \left[ \mathcal{D}(\Omega) \otimes \mathcal{C}_{per}^{\infty }(Y) /
\mathbb{R}
\right] \times \left[ \mathcal{D}(\Omega) \otimes \mathcal{C}_{per}^{\infty }(Y) \otimes \mathcal{C}_{per}^{\infty }(Z) /
\mathbb{R} 
\right] 
\end{equation} is dense in $\mathbb{F}_{0}^{1}L^{B}.$

\medskip

By hypotheses $\left( H_{1}\right)
-\left( H_{4}\right) $, it is easily seen that the following result holds

\begin{lemma}\label{mp}
\bigskip There exists a unique $u=\left( u_{0},u_{1},u_{2}\right) \in 
\mathbb{F}_{0}^{1}L^{B}$ such that $u$ solves \eqref{MP1}.

\end{lemma}

\subsection{Proof of Theorem \ref{main}}
This subsection is devoted to provide an application of reiterated two-scale convergence to the study of minimum problems involving integral functionals, i.e. to prove Theorem \ref{main}.
The proof will be achieved by means of several steps. First, following the same strategy in \cite{NNW}, (see also \cite{MT}) we regularize the integrands in order to get an approximating family of differentiable
integrands with some extra properties which will be detailed in the sequel. 

Let $f:\mathbb R^{N}\times
\mathbb R
^N\times \mathbb R^{nN} \to \mathbb R$ be such that $(H_1)-(H_4)$ hold.
Set
\begin{equation}
\label{fm}f_{m}:\left( y,z,\lambda \right)\in 
\mathbb{R}
^{N}\times
\mathbb{R}
^{N}\times 
\mathbb{R}
^{nN} \mapsto \int_{\mathbb R^{nN}} \theta _{m}\left( \eta
\right) f\left( y,z,\lambda -\eta \right) d\eta, 
\end{equation} where $\theta _{m}$ is a symmetric mollifier, namely $\theta_m \in \mathcal{D}\left( 
\mathbb{R}
^{nN}\right) \left( \text{integer }m\geq 1\right) $ with $0\leq \theta _{m},$
 ${\rm supp}\left( \theta _{m}\right) \subset \frac{1}{m}\overline{B}_{nN}(0,1)$, 
 $( B_{nN}(0,1)\text{ being the open unit ball in }\mathbb{R}
^{nN}$, and $\displaystyle{\int_{\overline{B_{nN}(0,1)}} \theta _{m}\left( \eta \right) d\eta =1}$. It is easily verified that
\begin{itemize}
\item[$(H_1) _{m}$] $f_{m}\left( \cdot,z,\lambda \right) $ is measurable
for every $(z,\lambda) \in 
\mathbb{R}
^{N}\times \mathbb{R}
^{nN}$ and $f_{m}\left( y,\cdot ,\lambda \right) $ is continuous for almost all $%
y\in 
\mathbb{R}
_{y}^{N};$
\item[$\left( H_{2}\right) _{m}$] $f_{m}\left( y,z,\cdot \right) $ is strictly convex
for almost all $\left( y,z\right) \in 
\mathbb R_y^N\times 
\mathbb{R}
_{z}^{N}$.
\item[$\left( H_{3}\right) _{m}$] There exists a constant $c>0$ such that:
\begin{equation*}
f_{m}\left( y,z,\lambda \right) \leq c\left( 1+b\left( \left\vert \lambda
\right\vert \right) \right) ,
\end{equation*}%
for every $(z,\lambda) \in \mathbb R \times
\mathbb{R}
^{nN},$ and for almost all $y\in 
\mathbb{R}
^{N}.$
\item[$\left( H_{4}\right) _{m}$] $f_{m}\left( \cdot,\cdot,\lambda \right) $\ is periodic
for all $\lambda \in
\mathbb{R}
^{nN}$\ 
\item[$\left( H_{5}\right) _{m}$] $\frac{\partial f_{m}}{\partial \lambda }\left(
y,z,\lambda \right) $ exists for all $\lambda \in 
\mathbb{R}
^{nN}$ and for almost all $\left(y,z\right) $ and there exist a constant $
c=c\left( m\right) >0$ such that:
$$\left\vert \frac{\partial f_{m}}{\partial
\lambda }\left( y,z,\lambda \right) \right\vert \leq c\left( m\right) \left(
1+b\left( \left\vert \lambda \right\vert \right) \right) $$ for all $\lambda
\in 
\mathbb R
^{nN}$ and for almost all $\left(y,z\right) \in 
\mathbb{R}^{N}\times \mathbb R^N.$
\end{itemize}

All the convergence results established in Proposition \ref{proposition3.1} and Corollary \ref{corollary3.1} for $f$, remain valid with $f_{m}$ .
Moreover for every \ $v\in L_{per}^{B}\left( \Omega \times Y\times Z\right)
^{nN}$, one has $f_{m}\left( \cdot,\cdot, v\right)
\rightarrow f\left( \cdot,\cdot, v\right) $ in $L^{1}\left( \Omega ;L_{per}^{1}\left(
Y\times Z\right) \right),$ as $m \to +\infty$.

The next result extends to the Orlicz setting an argument presented in \cite{NNW} to prove Corollary 2.10 therein.
\begin{proposition}
Let $\left( v_{\varepsilon }\right)$ be a sequence in $%
L^{B}\left( \Omega \right) ^{nN}$ which reiteratively  two-scale
converges (in each component) to $v\in L_{per}^{B}\left( \Omega \times Y\times
Z\right) ^{nN}$, then, for any integer $m\geq 1$, we have that there exists a constant $C'$ such that
\begin{equation*}
\iiint_{\Omega
\times Y\times Z}f_{m}\left( y,z,v\right) dxdydz-\frac{C'}{m}\leq \underset{\varepsilon
\rightarrow 0}{\lim \inf }\int_{\Omega }f\left( \frac{x}{\varepsilon },%
\frac{x}{\varepsilon ^{2}},v_{\varepsilon }\left( x\right) \right) dx.
\end{equation*}
\end{proposition}

\begin{proof}
Let $\left( v_{l}\right) _{l\geq 1}$ be a sequence in $\mathcal{D}(\Omega ;\mathbb R) \otimes \mathcal{C}_{per}^{\infty }(Y;
\mathbb R) \otimes \mathcal{C}_{per}^{\infty }( Z;
\mathbb R) $ such that $v_{l}\rightarrow v$ in $L_{per}^{B}\left( \Omega \times
Y\times Z\right) ^{nN}$ as $l\rightarrow \infty .$ The convexity and
differentiability of $f_{m}\left( y,z,\cdot\right) $
imply (for any integer $l\geq 1)$, 
\begin{align*}
\int_{\Omega }f_{m}\left( \frac{x}{%
\varepsilon },\frac{x}{\varepsilon ^{2}},v_{\varepsilon }\left( x\right)
\right) dx\geq \int_{\Omega }f_{m}\left( \frac{x}{\varepsilon },\frac{x}{%
\varepsilon ^{2}},v_{l}\left( x,\frac{x}{\varepsilon },\frac{x}{\varepsilon
^{2}}\right) \right) dx
\\
+\int_{\Omega }\frac{\partial f_{m}}{\partial \lambda }\left(\frac{x}{%
\varepsilon },\frac{x}{\varepsilon ^{2}},v_{l}\left( x,\frac{x}{\varepsilon }%
,\frac{x}{\varepsilon ^{2}}\right) \right) \cdot\left( v_{\varepsilon }\left(
x\right) -v_{l}\left(x,\frac{x}{\varepsilon },\frac{x}{\varepsilon ^{2}}%
\right) \right) dx.
\end{align*}

\noindent $(H_1)_m, (H_2)_m$ and $(H_5)_m$ guarantee that $x\longmapsto \frac{\partial
f_{m}}{\partial \lambda }\left(\cdot, \cdot,v_{l}\right) \in \mathcal{C}%
\left( \overline{\Omega };L_{per}^{\infty }\left( Y;\mathcal{C}%
_{per}^{\infty }\left( Z\right) \right) \right) $ hence, by Proposition \ref{proposition3.1}, it results
\begin{align*}
\underset{\varepsilon \rightarrow 0}{\lim }\int_{\Omega }\frac{\partial
f_{m}}{\partial \lambda }\left( \frac{x}{\varepsilon },\frac{x}{\varepsilon
^{2}},v_{l}\left( x,\frac{x}{\varepsilon },\frac{x}{\varepsilon ^{2}}\right)
\right) \cdot\left( v_{\varepsilon }\left( x\right) -v_{l}\left( x,\frac{x}{%
\varepsilon },\frac{x}{\varepsilon ^{2}}\right) \right) dx\\
=\iiint_{\Omega \times Y\times Z}\frac{\partial f_{m}}{\partial \lambda }%
\left(y,z,v_{l}\left( x,y,z\right) \right) \cdot\left( v\left( x,y,z\right)
-v_{l}\left( x,y,z\right) \right) dxdydz.
\end{align*}

Next, we observe that for a.e. $y$ and every $z, \lambda$ and a suitable positive constant $c$, one has
\begin{equation}\label{estimatefmf}
f_{m}\left( y,z,\lambda \right) \leq f\left( y,z,\lambda \ \right) +\frac{1}{%
	m}c\left( 1+b\left( 2\left( 1+\left\vert \lambda \ \right\vert \right)
\right) \right) .
\end{equation}
Indeed, for a.e. $y$, every $z,\lambda, \mu$, by \eqref{Blip}, 
\begin{align*}
f\left( y,z,\lambda \right) \leq f\left(y,z,\mu \right) +c\frac{B\left( 2\left( 1+\left\vert \lambda \right\vert +\left\vert \mu \right\vert
	\right) \right) }{1+\left\vert \lambda \right\vert +\left\vert \mu
	\right\vert }\left\vert \lambda -\mu \right\vert \\ \leq f\left(y,z,\mu
\right) +c\left( 1+b\left( 1+\left\vert \lambda \right\vert +\left\vert \mu
\right\vert \right) \right) \left\vert \lambda -\mu \right\vert .
\end{align*}
 Replacing $
\lambda $ by $\lambda -\eta $\ and $\mu $ by $\lambda $\ respectively, we 
obtain: 
\begin{align*}
f\left(y,z,\lambda -\eta\right) \leq f\left( y,z,\lambda\right) +c\left( 1+b\left( 1+\left\vert \lambda -\eta\right\vert
+\left\vert \lambda\right\vert \right) \right) \left\vert \eta\right\vert \\
\leq f\left( y,z,\lambda \right) +c\left( 1+b\left(
1+\left\vert \eta\right\vert +2\left\vert \lambda \right\vert \right)
\right) \left\vert \eta\right\vert.
\end{align*} 
Let $m>0,$ and assume $\left\vert
\eta \ \right\vert \leq \frac{1}{m}\leq 1,$ hence, $$f\left( y,z,\lambda
-\eta \right) \leq f\left( x,y,\lambda\right)+
c\left( 1+b\left( 2\left( 1+\left\vert \lambda \right\vert \right)
\right) \right) \frac{1}{m}.$$ 
Multiplying both side of the inequality, by $\theta_{m}$, we
get: 
\begin{align*}
f\left(y,z,\lambda -\eta\right) \theta _{m}\left( \eta \right)
\leq f\left(y,z,\lambda \right) \theta _{m}\left( \eta \right) +\frac{1}{m}c\left( 1+b\left( 2\left( 1+\left\vert \lambda \right\vert \right)
\right) \right) \theta _{m}\left( \eta \right) .
\end{align*}
Integration leads to \eqref{estimatefmf}.
Hence, given $v_\varepsilon$, we have 
\begin{align*}
f_{m}\left( \frac{x}{\varepsilon },\frac{x}{\varepsilon ^{2}}%
,v_{\varepsilon }\right) \leq f\left( \frac{x}{\varepsilon },\frac{x}{%
	\varepsilon ^{2}},v_{\varepsilon }\ \right) +\frac{1}{m}c\left( 1+b\left(
2\left( 1+\left\vert v_{\varepsilon }\right\vert \right) \right) \right) 
\end{align*}
thus 
\begin{align*}
\int_{\Omega }f_{m}\left( \frac{x}{\varepsilon },\frac{x}{\varepsilon
	^{2}},v_{\varepsilon }\right) dx\leq \int_{\Omega }f\left( \frac{x}{%
	\varepsilon },\frac{x}{\varepsilon ^{2}},v_{\varepsilon }\ \right) dx+\frac{%
	1}{m}C |\Omega| +\frac{c}{m}\int_{\Omega }\alpha \frac{%
	b\left( 2\left( 1+\left\vert v_{\varepsilon }\right\vert \right) \right) }{%
	\alpha }dx,
\\0<\alpha \leq 1
\end{align*}

\bigskip But $\alpha \frac{b\left( 2\left( 1+\left\vert v_{\varepsilon
	}\right\vert \right) \right) }{\alpha }\leq \widetilde{B}\left( \alpha
b\left( 2\left( 1+\left\vert v_{\varepsilon }\right\vert \right) \right)
\right) +B\left( \frac{1}{\alpha }\right) \leq \alpha \widetilde{B}\left(
b\left( 2\left( 1+\left\vert v_{\varepsilon }\right\vert \right) \right)
\right) +B\left( \frac{1}{\alpha }\right) $

Set $\Omega _{1}=\left\{ x\in \Omega :2\left( 1+\left\vert v_{\varepsilon
}\left( x\right) \right\vert \right) >t_{0}\right\} ,\Omega_2=\Omega
\backslash \Omega_{1}.$ 

Hence, we get 
\begin{align*}
\int_{\Omega }\alpha \frac{b\left( 2\left( 1+\left\vert
	v_{\varepsilon }\right\vert \right) \right) }{\alpha }dx\leq \int_{\Omega
}\alpha \widetilde{B}\left( b\left( 2\left( 1+\left\vert v_{\varepsilon
}\right\vert \right) \right) \right) dx+B\left( \frac{1}{\alpha }\right)|\Omega|\leq \\
\\\int_{\Omega _{1}}\alpha \widetilde{B}\left( b\left( 2\left( 1+\left\vert
v_{\varepsilon }\right\vert \right) \right) \right) dx+\int_{\Omega
	_{2}}\alpha \widetilde{B}\left( b\left( 2\left( 1+\left\vert v_{\varepsilon
}\right\vert \right) \right) \right) dx+B\left( \frac{1}{\alpha }\right)
| \Omega| \leq \\
|\Omega _{2}| \alpha \widetilde{B}\left( b\left(
t_{0}\right) \right) +B\left( \frac{1}{\alpha }\right) |\Omega| +\alpha \int_{\Omega_{1}}B\left( 4\left( 1+\left\vert v_{\varepsilon }\right\vert \right) \right) dx.
\end{align*}
\bigskip 

Let $C>1+\left\Vert 4\left( 1+\left\vert v_{\varepsilon
}\right\vert \right) \right\Vert _{B,\Omega }.$ Then $\int_{\Omega }B\left( 
\frac{4\left( 1+\left\vert v_{\varepsilon }\right\vert \right) }{C}\right)
dx\leq 1.$

Since  $B\left( 4\left( 1+\left\vert v_{\varepsilon }\right\vert
\right) \right) =B\left( C\frac{4\left( 1+\left\vert v_{\varepsilon}\right\vert \right) }{C}\right) \leq K\left( C\right) B\left( \frac{4\left(
	1+\left\vert v_{\varepsilon }\right\vert \right) }{C}\right)$ whenever $\frac{%
	4\left( 1+\left\vert v_{\varepsilon }\right\vert \right) }{C}\geq t_{0}.$

Set $\Omega _{3}=\left\{ x\in \Omega _{1}:\frac{4\left( 1+\left\vert v_{\varepsilon }\right\vert \right) }{C}\geq t_{0}\right\} ,\Omega
_{4}=\Omega _{1}\backslash \Omega _{3}.$

Hence 
\begin{align*}
&\int_{\Omega _{1}}B\left(
4\left( 1+\left\vert v_{\varepsilon }\right\vert \right) \right)
dx=\int_{\Omega _{4}}B\left( 4\left( 1+\left\vert v_{\varepsilon
}\right\vert \right) \right) dx+\int_{\Omega _{3}}B\left( 4\left(
1+\left\vert v_{\varepsilon }\right\vert \right) \right) dx\\
&\leq  |\Omega _{4}| B\left( Ct_{0}\right) +\int_{\Omega
	_{3}}B\left( 4\left( 1+\left\vert v_{\varepsilon }\right\vert \right)
\right) dx\leq |\Omega _{4}| B\left( Ct_{0}\right)
+\int_{\Omega _{3}}B\left( C\frac{4\left( 1+\left\vert v_{\varepsilon	}\right\vert \right) }{C}\right) dx\\
&\leq| \Omega _{4}| B\left( Ct_{0}\right) +K\left( C\right)
\int_{\Omega _{3}}B\left( \frac{4\left( 1+\left\vert v_{\varepsilon}\right\vert \right) }{C}\right) dx\leq | \Omega_4|B\left(
Ct_{0}\right) +K\left( C\right) \int_{\Omega }B\left( \frac{4\left(
	1+\left\vert v_{\varepsilon}\right\vert \right) }{C}\right) dx\\
&\leq |\Omega _{4}| B\left( Ct_{0}\right) +K\left( C\right)\int_{\Omega }B\left( \frac{4\left(
	1+\left\vert v_{\varepsilon}\right\vert \right) }{C}\right) dx.
\end{align*}

Since $\ B\in \triangle _{2}$, and $\left( v_{\varepsilon }\right) $ is bounded in $L^{B}\left( \Omega \right)$ it results that $\int_{\Omega }B\left(
4\left( 1+\left\vert v_{\varepsilon }\right\vert \right) \right) dx$ is also
bounded.

Then we have 
\begin{align*}
&\int_{\Omega }f_{m}\left( \frac{x}{\varepsilon },\frac{x}{%
	\varepsilon ^{2}},v_{\varepsilon }\right) dx\leq \int_{\Omega }f\left( 
\frac{x}{\varepsilon },\frac{x}{\varepsilon ^{2}},v_{\varepsilon }\right)
dx+\frac{1}{m}C|\Omega| +\\
&\frac{c}{m}\left( \alpha |\Omega| \widetilde{B}\left(
b\left( t_{0}\right) \right) +B\left( \frac{1}{\alpha }\right)|
\Omega| +\alpha \left(|\Omega _{4}|B\left(
Ct_{0}\right) +K\left( C\right) \right)\int_{\Omega }B\left( \frac{4\left(
	1+\left\vert v_{\varepsilon}\right\vert \right) }{C}\right) dx \right)\\
&\leq \int_{\Omega }f\left( 
\frac{x}{\varepsilon },\frac{x}{\varepsilon ^{2}},v_{\varepsilon }\right)
dx+\frac{1}{m}C',
\end{align*}
for a suitably big constant $C'$.
Thus 


\begin{align*}
\underset{\varepsilon \rightarrow 0}{\lim \inf }\int_{\Omega }f\left(\frac{x}{\varepsilon },\frac{x}{\varepsilon ^{2}},v_{\varepsilon }\left(
x\right) \right) dx\geq \iiint_{\Omega \times Y\times Z}f_{m}\left(
y,z,v_{l}\left( x,y,z\right) \right) dxdydz
\\
-\frac{C'}{m}+\iiint_{\Omega \times Y\times Z}\frac{\partial f_{m}}{\partial \lambda }%
\left(y,z,v_{l}\left( x,y,z\right) \right) \cdot\left( v\left( x,y,z\right)
-v_{l}\left( x,y,z\right) \right) dxdydz.
\end{align*}
Using \ $\left( H_{5}\right)
_{m}$\ we get 
\begin{align*}
\left\vert \iiint_{\Omega \times Y\times Z}\frac{\partial f_{m}%
}{\partial \lambda }\left(y,z,v_{l}\left( x,y,z\right) \right) \cdot \left(
v\left( x,y,z\right) -v_{l}\left( x,y,z\right) \right) dxdydz\right\vert 
\\
\leq c\left\Vert v-v_{l}\right\Vert _{B,\Omega \times Y\times Z}\cdot\left\Vert
1+b\left( v_{l}\right) \right\Vert _{\widetilde{B},\Omega \times Y\times Z}.
\end{align*}
Since $v_{l}\rightarrow v$ in $L_{per}^{B}\left( \Omega \times Y\times
Z\right) ^{nN}$ as $l\rightarrow \infty ,$ it follows that for $\delta >0$
arbitrarily fixed, there exists $l_{0}\in 
\mathbb N,$ such that 
\begin{equation*}\left\vert \iiint_{\Omega \times Y\times Z}\frac{\partial f_{m}}{%
\partial \lambda }\left(y,z,v_{l}\left( x,y,z\right) \right) \cdot\left(
v\left( x,y,z\right) -v_{l}\left( x,y,z\right) \right) dxdydz\right\vert 
\leq \delta 
\end{equation*} for all $l\geq l_{0}.$ Hence for all $l\geq l_{0},$ 
\begin{equation*}
\underset{\varepsilon \rightarrow 0}{\lim \inf }\int_{\Omega }f\left(%
\frac{x}{\varepsilon },\frac{x}{\varepsilon ^{2}},v_{\varepsilon }\left(
x\right) \right) dx\geq \iiint_{\Omega \times Y\times Z}f_{m}\left(
y,z,v_{l}\left( x,y,z\right) \right) dxdydz-\delta- \frac{C'}{m} ;
\end{equation*}
Now sending $l \to \infty$ we have
$$
\underset{\varepsilon \rightarrow 0}{\lim \inf }\int_{\Omega }f\left(%
\frac{x}{\varepsilon },\frac{x}{\varepsilon ^{2}},v_{\varepsilon }\left(
x\right) \right) dx\geq \iiint_{\Omega \times Y\times Z}f_{m}\left(
y,z,v\left( x,y,z\right) \right) dxdydz-\delta- \frac{C'}{m}.
$$
The arbritrariness of $\delta$, concludes the proof.
\end{proof}

Letting $m \to +\infty$, and replacing $v_\varepsilon$ by $Du_\varepsilon$, with $u_\varepsilon$ reiteratively two-scale convergent to $u(x,y,z):= u_0(x)+u_1(x,y)+u_2(x,y,z)$ in $W^1L^B(\Omega;\mathbb R^n)$, one obtains the following result:

\begin{corollary}\label{maincor}
Let $\left( u_{\varepsilon }\right) _{\varepsilon }$ be a sequence in $%
W_{0}^{1}L^{B}\left( \Omega;
\mathbb R ^n\right) $
reiteratively two-scale convergent to $
u=\left( u_{0},u_{1},u_{2}\right) \in \mathbb{F}_{0}^{1}L^{B}$. Then 
\begin{equation*}
\iiint_{\Omega \times Y\times Z}f\left( y,z,\mathbb{D}u\left( x,y,z\right)
\right) dxdydz\leq \underset{\varepsilon \rightarrow 0}{\lim \inf }%
\int_{\Omega }f\left(\frac{x}{\varepsilon },\frac{x}{\varepsilon ^{2}}%
,Du_{\varepsilon }\left( x\right) \right) dx,
\end{equation*}
where 
$\mathbb{D}u=Du_{0}+D_{y}u_{1}+D_{z}u_{2}$.
\end{corollary}

Now we are in position to put together all the previous results in order to prove our main result.

\begin{proof}[Proof of Theorem \ref{main}]
	
For every $\varepsilon$, let $u_{\varepsilon }$ be a minimizer of $F_{\varepsilon }$. Hypothesis $(H_4)$ guarantees that
\ $(u_{\varepsilon })_\varepsilon$ is bounded in $W_{0}^{1}L^{B}\left( \Omega ;
\mathbb{R}
\right) ^{n}$. On the other hand, since the real sequence $\left( F_{\varepsilon }\left(
u_{\varepsilon }\right) \right) _{\varepsilon >0}$ is bounded, we can extract a not relabelled subsequence, such that we have $\left( a\right) -\left( b\right) 
,$ in the statement, and $\underset{\varepsilon \rightarrow 0}{\lim }F_{\varepsilon }\left(
u_{\varepsilon }\right) $ hold. 

It remains to verify that 
 $u=\left( u_{0},u_{1},u_{2}\right) $ is the solution of the
minimization problem $\left( \ref{mp}\right) .$ Let $\phi =\left( \psi
_{0},\psi _{1},\psi _{2}\right) \in F_{0}^{\infty }$ with $\psi _{0}\in 
\mathcal{D}\left( \Omega \right) ^{n},\psi _{1}\in \left[ \mathcal{D}\left(
\Omega \right) \otimes \mathcal{C}_{per}^{\infty }\left( Y\right) /%
\mathbb{R}
\right] ^{n}$, $\psi _{2}\in \left[ \mathcal{C}_{0}^{\infty }\left( \Omega
\right) \otimes \mathcal{C}_{per}^{\infty }\left( Y\right) \otimes \mathcal{C%
}_{per}^{\infty }\left( Z\right) /
\mathbb{R}
\right] ^{n}$. Define $\phi _{\varepsilon }:=\psi
_{0}+\varepsilon \psi _{1}+\varepsilon ^{2}\psi _{2}.$ Then $\phi
_{\varepsilon }\in W_{0}^{1}L^{B}\left( \Omega ;%
\mathbb{R}
\right) ^{n}$ so that we have 
\begin{equation*}
\int_{\Omega }f\left( \frac{x}{\varepsilon },%
\frac{x}{\varepsilon ^{2}},Du_{\varepsilon }\left( x\right) \right) dx\leq
\int_{\Omega }f\left( \frac{x}{\varepsilon },\frac{x}{\varepsilon ^{2}}%
,D\phi _{\varepsilon }\left( x\right) \right) dx.
\end{equation*} 
Therefore, taking the
limit as $\varepsilon \rightarrow 0$, using \ the arbitrariness of $\phi $,
the density of $F_{0}^{\infty }$ in $\mathbb{F}_{0}^{1}L^{B}$ the above
inequality leads us to 
\begin{equation*}
\underset{\varepsilon \rightarrow 0}{\lim }%
\int_{\Omega }f\left( \frac{x}{\varepsilon },\frac{x}{\varepsilon ^{2}}%
,Du_{\varepsilon }\left( x\right) \right) dx\leq \underset{v\in \mathbb{F}%
_{0}^{1}L^{B}}{\inf }\iiint_{\Omega \times Y\times Z}f\left( y,z,\mathbb{D}%
v\left( x,y,z\right) \right) dxdydz.
\end{equation*} 
This inequality, together with Corollary \ref{maincor}, leads
to the equality
$$
\iiint_{\Omega \times Y\times Z}f\left( y,z,\mathbb{D}%
u\left( x,y,z\right) \right) dxdydz=\underset{v\in \mathbb{F}%
	_{0}^{1}L^{B}}{\inf }\iiint_{\Omega \times Y\times Z}f\left( y,z,\mathbb{D}%
v\left( x,y,z\right) \right) dxdydz.
$$ 
Since (\ref{MP1}) has a unique solution, we can conclude that
the whole sequence $\left( u_{\varepsilon }\right) _{\varepsilon}$
verifies $\left( a\right) -\left( b\right) $\ \ and the proof is completed.
\end{proof}

The following corollary recasts the above results in terms of $\Gamma$-convergence with respect to reiterated two-scale convergence, thus extending the result proven in the single scale case in \cite{FTNZIMSE}, (see \cite{DalMaso} for details about $\Gamma$-convergence). 
\begin{corollary}
	\label{maincor2}
	Let $\Omega$ and $f$ be as in Theorem \ref{main}. Then, for every $u=\left( u_{0},u_{1},u_{2}\right) \in \mathbb{F}_{0}^{1}L^{B}$, it results
	\begin{align}
	\inf\left\{\liminf_{\varepsilon \to 0}\int_\Omega f\left(\frac{x}{\varepsilon},\frac{x}{\varepsilon^2},  Du_\varepsilon\right)dx:u_\varepsilon \rightharpoonup u \hbox{ weakly reiteratively two-scale }\right\}=\nonumber\\
		\inf\left\{\limsup_{\varepsilon \to 0}\int_\Omega f\left(\frac{x}{\varepsilon},\frac{x}{\varepsilon^2},  Du_\varepsilon\right)dx:u_\varepsilon \rightharpoonup u \hbox{ weakly  reiteratively two-scale } \right\}=	\label{mainalign} \\
		\iiint_{\Omega \times Y \times Z}f(y,x, \mathbb D u(x,y,z))dxdydz, \nonumber
	\end{align}
	where  $\mathbb{D}u=Du_{0}+D_{y}u_{1}+D_{z}u_{2}$.
\end{corollary}
\begin{proof}[Proof]
The statement will be proven if we show that 
\begin{align*} 
\iiint_{\Omega \times Y \times Z}f(y,x, \mathbb D u(x,y,z))dxdydz\leq
\liminf_{\varepsilon \to 0}\int_\Omega f\left(\frac{x}{\varepsilon},\frac{x}{\varepsilon^2},D u_\varepsilon\right)dx,
\end{align*}
for any sequence $u_\varepsilon \rightharpoonup u\in \mathbb F^1_0L^B$ reiteratively two-scale, and we exhibit a sequence $\overline u_\varepsilon$ such that $\overline u_\varepsilon \rightharpoonup u\in \mathbb F^1_0L^B$ reiteratively two-scale, and
\begin{align*}
\limsup_{\varepsilon \to 0}\int_\Omega f\left(\frac{x}{\varepsilon},\frac{x}{\varepsilon^2},D \overline u_\varepsilon\right)dx\leq\iiint_{\Omega \times Y \times Z}f(y,x, \mathbb D u(x,y,z))dxdydz.
\end{align*}
The first inequality is consequence of Corollary \ref{maincor}. For what concerns the upper bound we preliminarily observe that a standard argument in the Orlicz setting allows us to consider, for any given ${\rm N}-$function $B$, a generating function $b$ such that $b$ is continuous and $B$ verifies the $\triangle_2$ condition near $0$.

\bigskip 
\noindent Now let $\phi_{\varepsilon }\left( x\right) :=\psi
_{0}+\varepsilon \psi _{1}\left( x,\frac{x}{\varepsilon }\right)
+\varepsilon ^{2}\psi _{1}\left( x,\frac{x}{\varepsilon },\frac{x}{%
	\varepsilon ^{2}}\right) $ for $x\in \Omega ,$ where
$\psi _{0}\in \mathcal{C}_{0}^{\infty }(\Omega),\psi
_{1}\in \left[ \mathcal{C}_{0}^{\infty }(\Omega) \otimes 
\mathcal{C}_{per}^{\infty }(Y) \right]$ and $\psi _{2}\in %
\left[ \mathcal{C}_{0}^{\infty }(\Omega) \otimes \mathcal{C}%
_{per}^{\infty }(Y) \otimes \mathcal{C}_{per}^{\infty }(Z) \right]$, then, 
\begin{align*}
\underset{\varepsilon \rightarrow 0}{\lim }\int_{\Omega }f\left( \frac{x}{%
	\varepsilon },\frac{x}{\varepsilon ^{2}},D\phi _{\varepsilon }\right) dx=
\\ 
\iiint_{\Omega \times Y\times Z}f\left( y,z,D\psi _{0}+D_{y}\psi
_{1}+D_{z}\psi _{2}\right) dxdydz.
\end{align*}

\noindent Let $\mathbb{F}^{1}L^{B}:=W^{1}L^{B}(\Omega) \times L_{D_{y}}^{B}\left( \Omega ;W_{\#}^{1}L^{B}(Y) \right) \times L_{D_{z}}^{B}\left( \Omega ;L_{per}^1\left(
Y;W_{\#}^{1}L^{B}(Z) \right) \right) $ where
$L_{D_{y}}^{B}\left( \Omega ;W_{\#}^{1}L^{B}(Y) \right)$, 
$L_{D_{z}}^{B}\left( \Omega ;L_{per}^1\left( Y;W_{\#}^{1}L^{B}(Z) \right) \right)$  
have been defined in \eqref{LF10}. Recalling also that $\mathbb{F}^{1}L^{B}$, equipped with the norm $\left\Vert u_{0}\right\Vert _{%
	\mathbb{F}^{1}L^{B}}=\left\Vert Du\right\Vert _{B,\Omega }+\left\Vert
D_{y}u_{1}\right\Vert _{B,\Omega \times Y}+\left\Vert D_{z}u_{2}\right\Vert
_{B,\Omega \times Y\times Z}$, $u_{0}=\left( u,u_{1},u_{2}\right) \in 
\mathbb{F}_{0}^{1}L^{B}$ is Banach space,
thanks to the density of $\mathcal{C}^{\infty}( \overline{\Omega }) $ in $W^{1}L^{B}(\Omega) ,$ of $\mathcal{C}_{per}^{\infty }(Y) /
\mathbb R
$ in $W_{\#}^{1}L_{per}^{B}(Y) $ and that of $\mathcal{C}_{per}^{\infty }(Y) \otimes \mathcal{C}_{per}^{\infty }(Z)/ 
\mathbb R
$ in $L_{per}^1\left( Y;W_{\#}^{1}L^{B}(Z) \right)$, the space $F^{\infty }:=\mathcal{C}^{\infty }( 
\overline{\Omega}) \times $ $\left[ \mathcal{D}(\Omega) \otimes \mathcal{C}_{per}^{\infty }(Y) /
\mathbb R
\right] \times $ $ \left[ \mathcal{D}(\Omega) \otimes \mathcal{C}_{per}^{\infty }(Y) \otimes \mathcal{C}_{per}^{\infty }(Z) /
\mathbb R
\right] $ \ is dense in $\mathbb{F}^{1}L^{B}.$ 

As above for $v_0=\left(
v,v_{1},v_{2}\right) \in \mathbb{F}^{1}L^{B}$ we denote by $\mathbb{D}%
v_0$ the sum $Dv+D_{y}v_{1}+D_{z}v_{2}.$ 
 
In view of the stated density, 
given $\delta >0,$ there exist $u_{\delta }\in \mathcal{C}^{\infty }( 
\overline{\Omega }) ,v_{\delta }\in \left[ \mathcal{D}(\Omega) \otimes \mathcal{C}_{per}^{\infty }(Y) /
\mathbb R
\right] ,w_{\delta }\in \left[ \mathcal{D}(\Omega) \otimes \mathcal{C}_{per}^{\infty }(Y) \otimes \mathcal{C}_{per}^{\infty }(Z)/
\mathbb{R}
\right] $ such that: 
\begin{equation*}
\left\Vert v-u_{\delta }\right\Vert _{W^{1}L^{B}(\Omega) }+\left\Vert v_{1}-v_{\delta }\right\Vert _{L^{1}\left(\Omega
	;W_{\#}^{1}L^{B}(Y) \right) }+\left\Vert v_{2}-w_{\delta }\right\Vert _{L^{1}\left(
	\Omega ;L_{per}^{B}\left( Y;W_{\#}^{1}L^{B}(Z) \right) \right) }<\delta .
\end{equation*}%
For every $\delta, \varepsilon >0 $ and for every $x\in \Omega ,$ define
$u_{\delta ,\varepsilon }\left( x\right) =:u_{\delta }\left( x\right)
+\varepsilon v_{\delta }\left( x,\frac{x}{\varepsilon }\right) +\varepsilon
^{2}w_{\delta }\left( x,\frac{x}{\varepsilon },\frac{x}{\varepsilon ^{2}}%
\right)$. It results that 
\begin{eqnarray*}
	D_{x}u_{\delta ,\varepsilon }\left( x\right)  &=&D_{x}u_{\delta }\left(
	x\right) +\varepsilon D_{x}v_{\delta }\left( x,\frac{x}{\varepsilon }\right)
	+\varepsilon ^{2}D_{x}w_{\delta }\left( x,\frac{x}{\varepsilon },\frac{x}{%
		\varepsilon ^{2}}\right) +D_{y}v_{\delta }\left( x,\frac{x}{\varepsilon }%
	\right) +\\
&&\varepsilon D_{y}w_{\delta }\left( x,\frac{x}{\varepsilon },\frac{x%
	}{\varepsilon ^{2}}\right) 
	+D_{z}w_{\delta }\left( x,\frac{x}{\varepsilon },\frac{x}{\varepsilon ^{2}}%
	\right). 
\end{eqnarray*}%
As immediate consequence, for $\delta $ fixed, 
\begin{equation*}
\begin{array}{ll}
u_{\delta ,\varepsilon
}\rightarrow u_{\delta } \hbox{ in } L^{B}\left( \Omega \right),
\\D_{x}u_{\delta
	,\varepsilon }{\rightarrow }D_{x}u_{\delta }+D_{y}v_{\delta
}+D_{z}w_{\delta } \hbox{ strongly reiteratively two-scale in } L_{per}^{B}\left( \Omega \times Y\times
Z\right),
\end{array}
\end{equation*}
as $\varepsilon \to 0$.

Next, setting 
\begin{equation*}
 c_{\delta ,\varepsilon }=:\left\Vert
u_{\delta ,\varepsilon }-v\right\Vert _{W^{1}L^{B}(\Omega)}+\left|
\left\Vert Du_{\delta
	,\varepsilon }\right\Vert_{L^B(\Omega)}-\left\Vert Dv+ D_{y}v_{1}+D_{y}v_{2}\right\Vert_{L^B(\Omega \times Y\times Z)} \right|,
\end{equation*}
using the above density results:\begin{equation*}
\underset{\delta \rightarrow 0}{\lim }\underset{\varepsilon
	\rightarrow 0}{\lim }c_{\delta ,\varepsilon }=0.
\end{equation*}%
Then, via diagonalization,  we can construct a sequence $\delta \left(
	\varepsilon \right) \rightarrow 0,$ as $\varepsilon \rightarrow 0$ and such
that: 
\begin{itemize}
	\item[(i)] $\underset{\delta \left( \varepsilon \right) \rightarrow 0}{\lim }%
c_{\delta(\varepsilon),\varepsilon }=0.$
\item[(ii)] $u_{\delta(\varepsilon),\varepsilon }\rightarrow v$ in $%
L^{B}\left( \Omega \right)$, 
\item[(iii)] $Du_{\delta(\varepsilon),\varepsilon}
	\rightharpoonup D_{x}v+D_{y}v_{1}+D_{z}v_{2}$ strongly reiteratively in $L_{per}^{B}\left(
\Omega \times Y\times Z\right) .$
\end{itemize} 
In particular, it follows that \ $Du_{\delta(\varepsilon)
	,\varepsilon }\rightharpoonup D_{x}v$ weakly in $L^{B}\left( \Omega \right) ,
$ and 
\begin{align*}
\underset{\varepsilon \rightarrow 0}{\lim }\int_{\Omega }f\left( \frac{x}{%
	\varepsilon },\frac{x}{\varepsilon ^{2}},Du_{\delta(\varepsilon) ,\varepsilon
}(x)\right) dx=\\ 
\iiint_{\Omega \times Y\times Z}f\left(
y,z,D_{x}v+D_{y}v_{1}+D_{z}v_{2}\right) dxdydz.
\end{align*}
Since the above construction can be performed for every triple $(v,v_1,v_2)\in \mathbb F^1L^B$, it is enough to repeat the construction for $u_0=(u,u_1,u_2)\in \mathbb F^1_0L^B$ as claimed.
\end{proof}



\begin{remark}
	It is worth to observe that the result in Corollary \ref{maincor2} holds, with the exact same proof under weaker assumptions than those in Theorem \ref{main}: namely $(H_2)$ can be replaced by convexity, and in $(H_4)$ it is not crucial to have $f$ non-negative, it is enough to have a bound from below.
	Moreover the same proof can be performed if $u_\varepsilon$ and $u$ are vector valued and not just scalar valued functions.	\end{remark}

\section{Appendix}
Here we present the proof Proposition \ref{prop2.1} which establishes the equivalence between the norms $\left\Vert \cdot\right\Vert _{B,Y \times Z }$ and $ \left\Vert
\cdot \right\Vert _{\Xi ^{B}\left(
	\mathbb R_y^N;\mathcal C_b\left( 
	\mathbb R_z^N\right) \right) }$  in $\mathfrak{X}_{per}^{B}\left( 
\mathbb R_y^N;\mathcal C_b\right) .$

\begin{proof}[Proof of Proposition \ref{prop2.1}]
	The inclusion is a direct consequence of \ the definition, and clearly every element in $L^B_{per}(Y\times Z)$, can be obtained as limit in $\|\cdot\|_{B, Y \times Z}$ norm of sequences in $\mathcal C_{per}(Y\times Z)$.
		
	On the other hand, by the very defintion of $\mathfrak{X}_{per}^{B}(
	\mathbb R_y^N;\mathcal C_b) $,  $v\in 
	\mathfrak{X}_{per}^{B}(
	\mathbb R_y^N;\mathcal C_b) $ if and only if there exist $(
	v_{n})_{n\in 
		\mathbb N
	}\in \mathcal{C}_{per}\left( Y\times Z\right) $ such that $\left(
	v_{n}\right) _{n\in 
		\mathbb{N}
	}$ converge to $v$ for the norm $\left\Vert \cdot\right\Vert _{\Xi ^{B}\left( 
		\mathbb{R}
		_{y}^{N};{\mathcal C}_b\left(
		\mathbb{R}
		_{z}^{N}\right) \right) }.$
	
	Thus for every  $w\in \mathfrak{X}_{per}^{B}\left( 
	\mathbb{R}
	_{y}^{N};\mathcal C_b\right) $
	there exist $\left( w_{n}\right) _{n\in 
		\mathbb{N}
	}\subset \mathcal{C}_{per}\left( Y\times Z\right) ,$such that as $%
	n\rightarrow \infty ,w_{n}\rightarrow w$ in $\Xi ^{B}\left( 
	\mathbb{R}
	_{y}^{N};\mathcal C_b\left( 
	\mathbb{R}
	_{z}^{N}\right) \right) .$ 
	
	\noindent	We claim that for every $u \in C_{per}(Y\times Z)$, it results $\left\Vert u\right\Vert _{B,Y \times Z
	}\leq \left\Vert u\right\Vert _{\Xi ^{B}\left( 
		\mathbb{R}
		_{y}^{N};\mathcal C_b \right) }$. From the claim it follows that
	
	$\left\Vert w_{n}-w_{m}\right\Vert _{B,Y \times Z }\leq \left\Vert
	w_{n}-w_{m}\right\Vert _{\Xi ^{B}\left( 
		\mathbb{R}
		_{y}^{N};\mathcal C_b\left( 
		\mathbb R
		_{z}^{N}\right) \right) },$ for all $m,n\in 
	\mathbb N.$ Therefore $\left( w_{n}\right) _{n\in 
		\mathbb N }$ is a Cauchy sequence in $\mathfrak{X}_{per}^{B}\left( 
	\mathbb R
	_{y}^{N}\times 
	\mathbb R_{z}^{N}\right) $ and in $\mathfrak{X}_{per}^{B}\left( 
	\mathbb R_{y}^{N};\mathcal C_b\right) .$ Hence there exist $w^{1}\in \mathfrak{X}%
	_{per}^{B}\left(
	\mathbb R_{y}^{N}\times 
	\mathbb R_{z}^{N}\right) ,w^{2}\in \mathfrak{X}_{per}^{B}\left( 
	\mathbb R
	_{y}^{N};\mathcal C_b\right) $ such that
	\begin{equation*}
	\underset{n\rightarrow \infty }{\lim }\left\Vert w_{n}-w^{1}\right\Vert
	_{B,Y \times Z }=\underset{n\rightarrow \infty }{\lim }\left\Vert
	w_{n}-w^{2}\right\Vert _{\Xi ^{B}\left( 
		\mathbb R
		_{y}^{N};{\mathcal C}_b\left( 
		\mathbb R_z^N\right) \right) }=0.
	\end{equation*}
	Moreover the passage to the limit guarantees that
	$\left\Vert w^{1}\right\Vert _{B,Y \times Z }\leq \left\Vert w^{2}\right\Vert
	_{\Xi ^{B}\left( 
		\mathbb R
		_{y}^{N};\mathcal C_b\left( 
		\mathbb R
		_{z}^{N}\right) \right) }$.  It is also clear, considering the convergence in the sense of distributions, that  $w^{1}=w^{2}$.
	
	It remains to prove the claim. To this end,
	let $u,v\in \mathcal{C}_{per}\left( Y\times Z\right) ;$ we have%
	\begin{align*}
	\left\vert \int_{B_{N}\left( 0,1\right) }u\left( \frac{x}{\varepsilon },%
	\frac{x}{\varepsilon ^{2}}\right) v\left( \frac{x}{\varepsilon },\frac{x}{%
		\varepsilon ^{2}}\right) dx\right\vert \leq \int_{B_{N}\left( 0,1\right)
	}\left\Vert u\left( \frac{x}{\varepsilon },\cdot \right) \right\Vert _{\infty
	}\left|v\left( \frac{x}{\varepsilon },\frac{x}{\varepsilon ^{2}}\right)\right| dx\leq 
	\\ 
	2\left\Vert v^{\varepsilon }\right\Vert _{\widetilde{B},B_{N}\left(
		0,1\right) }\left\Vert u\right\Vert _{\Xi ^{B}\left( 
		\mathbb R
		_{y}^{N};{\mathcal C}_b( 
		\mathbb R
		_{z}^{N}) \right) }.
	\end{align*} 
	Passing to limit, as $\varepsilon \to 0$, we obtain: 
	\begin{align*} 
	\left\vert \int_{Y\times Z}u\left( y,z\right) v\left( y,z\right)
	dydz\right\vert \leq 2\left\Vert v\right\Vert _{\widetilde{B},Y\times
		Z}\left\Vert u\right\Vert _{\Xi ^{B}\left( 
		\mathbb R
		_{y}^{N};{\mathcal C}_b(\mathbb R^N_z) \right) }. 
	\end{align*} 
	Using the density of $\mathcal{C}_{per}\left( Y\times Z\right) $ in $L_{per}^{%
		\widetilde{B}}\left( Y\times Z\right) $ we obtain (with the topology of the
	norm) 
	\begin{align*}
	\left\vert \int_{Y\times Z}u\left( y,z\right) v\left( y,z\right)
	dydz\right\vert \leq 2\left\Vert v\right\Vert _{\widetilde{B},Y\times
		Z}\left\Vert u\right\Vert _{\Xi ^{B}\left( 
		\mathbb{R}
		_{y}^{N};{\mathcal C}_b\left( 
		\mathbb{R}
		_{z}^{N}\right) \right) }, 
	\end{align*} 
	for all $v\in L_{per}^{\widetilde{B}}\left(
	Y\times Z\right)$. 
	Thus $\left\Vert u\right\Vert _{B,Y\times Z}\leq 2\left\Vert u\right\Vert
	_{\Xi ^{B}\left( 
		\mathbb{R}
		_{y}^{N};{\mathcal C}_b\left( 
		\mathbb{R}
		_{z}^{N}\right) \right) },$ for all $u\in \mathcal{C}_{per}\left( Y\times
	Z\right) $,
	and we get the result for all $u\in \mathfrak{X}^B_{per}\left(
	\mathbb R
	_{y}^{N};{\mathcal C}_b\right) $, via standard density arguments.
\end{proof}
\color{black}
\section{Acknowledgements}

This paper has been written during the visit of J.F.T. at Dipartimento di Ingegneria Industriale  (INdAM unit) at University of
Salerno. The authors gratefully acknowledge the supports of the INdAM-ICTP Research in pairs programme. E. Z. is a member of INdAM-GNAMPA.

\smallskip

\end{document}